\documentclass[11pt,twoside,a4paper]{article}
\usepackage{makeidx}
\usepackage[dvips]{graphicx}
\usepackage{amscd}
\usepackage{amsmath}
\usepackage{amssymb}
\usepackage{latexsym}
\usepackage[latin1]{inputenc}
\usepackage[T1]{fontenc}   
\usepackage[english]{babel}
\newcommand{\tun}{\begin{picture}(5,0)(-2,-1)
\put(0,0){\circle*{2}}
\end{picture}}

\newcommand{\tdeux}{\begin{picture}(7,7)(0,-1)
\put(3,0){\circle*{2}}
\put(3,0){\line(0,1){5}}
\put(3,5){\circle*{2}}
\end{picture}}

\newcommand{\ttroisun}{\begin{picture}(15,8)(-5,-1)
\put(3,0){\circle*{2}}
\put(-0.65,0){$\vee$}
\put(6,7){\circle*{2}}
\put(0,7){\circle*{2}}
\end{picture}}
\newcommand{\ttroisdeux}{\begin{picture}(5,12)(-2,-1)
\put(0,0){\circle*{2}}
\put(0,0){\line(0,1){5}}
\put(0,5){\circle*{2}}
\put(0,5){\line(0,1){5}}
\put(0,10){\circle*{2}}
\end{picture}}

\newcommand{\tquatreun}{\begin{picture}(15,12)(-5,-1)
\put(3,0){\circle*{2}}
\put(-0.65,0){$\vee$}
\put(6,7){\circle*{2}}
\put(0,7){\circle*{2}}
\put(3,7){\circle*{2}}
\put(3,0){\line(0,1){7}}
\end{picture}}
\newcommand{\tquatredeux}{\begin{picture}(15,18)(-5,-1)
\put(3,0){\circle*{2}}
\put(-0.65,0){$\vee$}
\put(6,7){\circle*{2}}
\put(0,7){\circle*{2}}
\put(0,14){\circle*{2}}
\put(0,7){\line(0,1){7}}
\end{picture}}
\newcommand{\tquatretrois}{\begin{picture}(15,18)(-5,-1)
\put(3,0){\circle*{2}}
\put(-0.65,0){$\vee$}
\put(6,7){\circle*{2}}
\put(0,7){\circle*{2}}
\put(6,14){\circle*{2}}
\put(6,7){\line(0,1){7}}
\end{picture}}
\newcommand{\tquatrequatre}{\begin{picture}(15,18)(-5,-1)
\put(3,5){\circle*{2}}
\put(-0.65,5){$\vee$}
\put(6,12){\circle*{2}}
\put(0,12){\circle*{2}}
\put(3,0){\circle*{2}}
\put(3,0){\line(0,1){5}}
\end{picture}}
\newcommand{\tquatrecinq}{\begin{picture}(9,19)(-2,-1)
\put(0,0){\circle*{2}}
\put(0,0){\line(0,1){5}}
\put(0,5){\circle*{2}}
\put(0,5){\line(0,1){5}}
\put(0,10){\circle*{2}}
\put(0,10){\line(0,1){5}}
\put(0,15){\circle*{2}}
\end{picture}}

\newcommand{\tcinqsept}{\begin{picture}(15,8)(-5,-1)
\put(3,0){\circle*{2}}
\put(-0.65,0){$\vee$}
\put(6,7){\circle*{2}}
\put(0,7){\circle*{2}}
\put(2.35,7){$\vee$}
\put(3,14){\circle*{2}}
\put(9,14){\circle*{2}}
\end{picture}}

\newcommand{\tcinqneuf}{\begin{picture}(15,26)(-5,-1)
\put(3,0){\circle*{2}}
\put(-0.65,0){$\vee$}
\put(6,7){\circle*{2}}
\put(0,7){\circle*{2}}
\put(6,14){\circle*{2}}
\put(6,7){\line(0,1){7}}
\put(6,21){\circle*{2}}
\put(6,14){\line(0,1){7}}
\end{picture}}
\newcommand{\tcinqdix}{\begin{picture}(15,19)(-5,-1)
\put(3,5){\circle*{2}}
\put(-0.5,5){$\vee$}
\put(6,12){\circle*{2}}
\put(0,12){\circle*{2}}
\put(3,0){\circle*{2}}
\put(3,0){\line(0,1){12}}
\put(3,12){\circle*{2}}
\end{picture}}
\newcommand{\tcinqonze}{\begin{picture}(15,26)(-5,-1)
\put(3,5){\circle*{2}}
\put(-0.65,5){$\vee$}
\put(6,12){\circle*{2}}
\put(0,12){\circle*{2}}
\put(3,0){\circle*{2}}
\put(3,0){\line(0,1){5}}
\put(0,12){\line(0,1){7}}
\put(0,19){\circle*{2}}
\end{picture}}
\newcommand{\tcinqdouze}{\begin{picture}(15,26)(-5,-1)
\put(3,5){\circle*{2}}
\put(-0.65,5){$\vee$}
\put(6,12){\circle*{2}}
\put(0,12){\circle*{2}}
\put(3,0){\circle*{2}}
\put(3,0){\line(0,1){5}}
\put(6,12){\line(0,1){7}}
\put(6,19){\circle*{2}}
\end{picture}}
\newcommand{\tcinqtreize}{\begin{picture}(5,26)(-2,-1)
\put(0,0){\circle*{2}}
\put(0,0){\line(0,1){7}}
\put(0,7){\circle*{2}}
\put(0,7){\line(0,1){7}}
\put(0,14){\circle*{2}}
\put(-3.65,14){$\vee$}
\put(-3,21){\circle*{2}}
\put(3,21){\circle*{2}}
\end{picture}}
\newcommand{\tcinqquatorze}{\begin{picture}(9,26)(-5,-1)
\put(0,0){\circle*{2}}
\put(0,0){\line(0,1){5}}
\put(0,5){\circle*{2}}
\put(0,5){\line(0,1){5}}
\put(0,10){\circle*{2}}
\put(0,10){\line(0,1){5}}
\put(0,15){\circle*{2}}
\put(0,15){\line(0,1){5}}
\put(0,20){\circle*{2}}
\end{picture}}

\newcommand{\tdun}[1]{\begin{picture}(10,5)(-2,-1)
\put(0,0){\circle*{2}}
\put(3,-2){\tiny #1}
\end{picture}}

\newcommand{\tddeux}[2]{\begin{picture}(12,5)(0,-1)
\put(3,0){\circle*{2}}
\put(3,0){\line(0,1){5}}
\put(3,5){\circle*{2}}
\put(6,-2){\tiny #1}
\put(6,3){\tiny #2}
\end{picture}}

\newcommand{\tdtroisun}[3]{\begin{picture}(20,12)(-5,-1)
\put(3,0){\circle*{2}}
\put(-0.65,0){$\vee$}
\put(6,7){\circle*{2}}
\put(0,7){\circle*{2}}
\put(5,-2){\tiny #1}
\put(9,5){\tiny #2}
\put(-5,5){\tiny #3}
\end{picture}}
\newcommand{\tdtroisdeux}[3]{\begin{picture}(12,12)(-2,-1)
\put(0,0){\circle*{2}}
\put(0,0){\line(0,1){5}}
\put(0,5){\circle*{2}}
\put(0,5){\line(0,1){5}}
\put(0,10){\circle*{2}}
\put(3,-2){\tiny #1}
\put(3,3){\tiny #2}
\put(3,9){\tiny #3}
\end{picture}}

\newcommand{\tdquatreun}[4]{\begin{picture}(20,12)(-5,-1)
\put(3,0){\circle*{2}}
\put(-0.6,0){$\vee$}
\put(6,7){\circle*{2}}
\put(0,7){\circle*{2}}
\put(3,7){\circle*{2}}
\put(3,0){\line(0,1){7}}
\put(5,-2){\tiny #1}
\put(8.5,5){\tiny #2}
\put(1,10){\tiny #3}
\put(-5,5){\tiny #4}
\end{picture}}
\newcommand{\tdquatredeux}[4]{\begin{picture}(20,20)(-5,-1)
\put(3,0){\circle*{2}}
\put(-.65,0){$\vee$}
\put(6,7){\circle*{2}}
\put(0,7){\circle*{2}}
\put(0,14){\circle*{2}}
\put(0,7){\line(0,1){7}}
\put(5,-2){\tiny #1}
\put(9,5){\tiny #2}
\put(-5,5){\tiny #3}
\put(-5,12){\tiny #4}
\end{picture}}

\newcommand{\tdquatrequatre}[4]{\begin{picture}(20,14)(-5,-1)
\put(3,5){\circle*{2}}
\put(-.65,5){$\vee$}
\put(6,12){\circle*{2}}
\put(0,12){\circle*{2}}
\put(3,0){\circle*{2}}
\put(3,0){\line(0,1){5}}
\put(6,-3){\tiny #1}
\put(6,4){\tiny #2}
\put(9,12){\tiny #3}
\put(-5,12){\tiny #4}
\end{picture}}

\input{xy}
\xyoption{all}

\newcommand{\D}{\mathcal{D}}
\newcommand{\h}{\mathbf{H}}

\newcommand{\F}{\mathcal{F}}
\renewcommand{\S}{\mathfrak{S}}
\newcommand{\FF}{\mathbb{F}}
\newcommand{\GG}{\mathbb{G}}
\newcommand{\HH}{\mathbb{H}}
\newcommand{\tdelta}{\tilde{\Delta}}
\renewcommand{\vec}[1]{\boldsymbol{#1}}
\newcommand{\FQSym}{\mathbf{FQSym}}
\newcommand{\PQSym}{\mathbf{PQSym}}
\renewcommand{\P}{\mathbb{P}}
\newcommand{\mmodels}{\mid \hspace{-.2mm} \models}

\newtheorem{defi}{\indent Definition}
\newtheorem{lemma}[defi]{\indent Lemma}
\newtheorem{cor}[defi]{\indent Corollary}
\newtheorem{theo}[defi]{\indent Theorem}
\newtheorem{prop}[defi]{\indent Proposition}

\newenvironment{proof}{{\bf Proof.}}{\hfill $\Box$}

\title{Ordered forests and parking functions}
\author{L. Foissy \\
\\
{\small{\it Laboratoire de Mathématiques, Université de Reims}}\\
\small{{\it Moulin de la Housse - BP 1039 - 51687 REIMS Cedex 2, France}}\\
\small{e-mail : loic.foissy@univ-reims.fr}}

\date{}

\begin{document}

\maketitle

ABSTRACT. We prove that the Hopf algebra of parking functions and the Hopf algebra of ordered forests are isomorphic,
using a rigidity theorem for a particular type of bialgebras.\\


AMS CLASSIFICATION. 05C05, 16W30.

\tableofcontents

\section*{Introduction}

The Connes-Kreimer Hopf algebra of rooted trees is described in \cite{Connes}, in a context of Quantum Fields Theory: 
it is used to treat the Renormalization procedure. This Hopf algebra is generated by the set of rooted trees, and its coproduct is given by admissible cuts.
Other Hopf algebras of trees are obtained from this one by giving additional structures to the rooted trees. For example, adding planar datas,
one obtains the Hopf algebra of planar trees $\h_p$ and its decorated versions $\h_p^\D$ \cite{Foissy3,Holtkamp};
adding a total order on the vertices, one obtains the Hopf algebra of ordered forests $\h_o$ and its Hopf subalgebra $\h_{ho}$,
generated by heap-ordered forests (that is to say that the total order of the vertices is compatible with the oriented graph structure of the forests).
These two Hopf algebras appeared in \cite{FoissyUnt} in a probabilistic context, in order to define rough paths.
The main point of the construction is a Hopf algebra morphism $\Theta$ from $\h_o$ to the Hopf algebra of free quasi-symmetric functions $\FQSym$
\cite{Duchamp,Malvenuto}, also known as the Malvenuto-Reutenauer Hopf algebra of permutations; the restriction of $\Theta$ to $\h_{ho}$
is an isomorphism of Hopf algebras.\\

Our aim here is an algebraic study of $\h_o$. In particular, $\h_o$ and the Hopf algebra of parking functions $\PQSym$ \cite{Novelli2,Novelli1}
have the same Poincaré-Hilbert series: we prove here that they are isomorphic.
We also combinatorially define a symmetric Hopf pairing on $\h_o$, mimicking the Hopf pairing on the non-commutative Connes-Kreimer Hopf algebras.
It turns out that this pairing is degenerate, and its kernel is the kernel of the Hopf algebra morphism $\Theta$. Moreover,
$\Theta$ induces an isometry  from $\h_o/Ker(\Theta)$ or from $\h_{ho}$ to $\FQSym$.

In order to prove the isomorphism of $\h_o$ and $\PQSym$, we introduce the notion of $Dup$-$Dend$ bialgebra.
A duplicial algebra is an algebra with two associative (non unitary) products $m$ and $\nwarrow$, such that $(xy)\nwarrow z=x(y \nwarrow z)$
for all $x,y,z$. This type of algebra, studied in \cite{Loday2}, naturally appears in Quantum Electrodynamics, see \cite{Frabetti2,Frabetti1}.
A dendriform coalgebra, notion dual to the notion of dendriform algebra \cite{Loday,Ronco} is a coassociative (non counitary) coalgebra $C$, 
whose coproduct $\tdelta$ can be written $\Delta_\prec+\Delta_\succ$, such that $(C,\Delta_\succ,\Delta_\prec)$ is a bicomodule over $C$.
A $Dup$-$Dend$ bialgebra is both a duplicial algebra and a dendriform coalgebra, with compatibilities between the two products
and the two coproducts given by equations (\ref{E3}) and (\ref{E4}) of this text. We here prove that the augmentation ideals of $\h_o$, $\PQSym$ and $\h_p^\D$
are $Dup$-$Dend$ bialgebras; moreover, the augmentation ideals of $\h_{ho}$ and $\FQSym$ are sub-$Dup$-$Dend$ bialgebras
of respectively $\h_o$ and $\PQSym$, and $\Theta$ is a morphism of $Dup$-$Dend$ bialgebras.

We then observe that $\h_p^\D$ is the free duplicial algebra generated by the set $\D$. We finally prove a rigidity theorem {\it à la Loday},
which says that any graded, connected $Dup$-$Dend$ bialgebra is a free duplicial algebra, so is isomorphic
to a $\h_p^\D$ as a Hopf algebra. Manipulating formal series, we deduce that $\h_o$ and $\PQSym$ are isomorphic to the same $\h_p^\D$,
and are therefore isomorphic.\\

This paper is organised as follows: the two first sections are dedicated to reminders on respectively the Hopf algebras of trees,
and the Hopf algebras of permutations and parking functions,  $\FQSym$ and $\PQSym$.
The pairing on $\h_o$ is introduced and studied in the third section and the last part of the text deals with $Dup$-$Dend$ bialgebras,
the rigidity theorem and the existence of an isomorphism between $\h_o$ and $\PQSym$.\\

{\bf Notations}. \begin{enumerate}
\item $K$ is a commutative field. Any vector space, algebra, coalgebra,\ldots of this text will be taken over $K$.
\item Let $A=(A,m,\Delta,1,\varepsilon,S)$ be a Hopf algebra. The augmentation ideal $Ker(\varepsilon)$ of $A$ will be denoted by $A_+$.
We give $A_+$ a coassociative, but not counitary, coproduct $\tdelta$ defined by $\tdelta(x)=\Delta(x)-x \otimes 1-1\otimes x$ for all $x \in A_+$.
\end{enumerate}

\section{Four Hopf algebras of forests}

\subsection{The Connes-Kreimer Hopf algebra of rooted trees}

We briefly recall the construction of the Connes-Kreimer Hopf algebra of rooted trees \cite{Connes}.
A {\it rooted tree} is a finite tree with a distinguished vertex called the {\it root} \cite{Stanley}. 
A {\it rooted forest} is a finite graph $\FF$ such that any connected component of $\FF$ is a rooted tree.
The set of vertices of the rooted forest $\FF$ is denoted by $V(\FF)$.
The {\it degree} of a forest $\FF$ is the number of its vertices. The set of rooted forests of degree $n$ will be denoted by $\F(n)$.\\

For example:
\begin{eqnarray*}
\F(0)&=&\{1\},\\
\F(1)&=&\{\tun\},\\
\F(2)&=&\{\tun\tun,\tdeux\},\\
\F(3)&=&\{\tun\tun\tun,\tdeux\tun,\ttroisun,\ttroisdeux\},\\
\F(4)&=&\{\tun\tun\tun\tun,\tdeux\tun\tun,\tdeux\tdeux,\ttroisun\tun,\ttroisdeux\tun,\tquatreun,\tquatredeux,\tquatrequatre,\tquatrecinq\}.
\end{eqnarray*}

Let $\FF$ be a rooted forest. The edges of $\FF$ are oriented downwards (from the leaves to the roots). If $v,w \in V(\FF)$, we shall denote by
$v \rightarrow w$ if there is an edge in $\FF$ from $v$ to $w$ and $v \twoheadrightarrow w$ if there is an oriented path from $v$ to $w$ in $\FF$.
By convention, $v \twoheadrightarrow v$ for any $v \in V(\FF)$.

Let $\vec{v}$ be a subset of $V(\FF)$. We shall say that $\vec{v}$ is an admissible cut of $\FF$, 
and we shall write $\vec{v} \models V(\FF)$, if $\vec{v}$ is totally disconnected, that is to say that
$v \twoheadrightarrow w \hspace{-.7cm} / \hspace{.7cm}$ for any couple $(v,w)$ of two different elements of $\vec{v}$.
If $\vec{v} \models V(\FF)$, we denote by $Lea_{\vec{v}}\FF$ the rooted sub-forest of $\FF$ obtained by keeping
only the vertices above $\vec{v}$, that is to say $\{ w \in V(\FF), \: \exists v \in \vec{v}, \:w \twoheadrightarrow v \}$.
Note that $\vec{v} \subseteq Lea_{\vec{v}}\FF$. We denote by $Roo_{\vec{v}}\FF$ the rooted sub-forest obtained by keeping the other vertices.

In particular, if $\vec{v}=\emptyset$, then $Lea_{\vec{v}}\FF=1$ and $Roo_{\vec{v}}\FF=\FF$: this is the {\it empty cut} of $\FF$.
If $\vec{v}$ contains all the roots of $\FF$, then it contains only the roots of $\FF$, $Lea_{\vec{v}}\FF=\FF$ and
$Roo_{\vec{v}}\FF=1$: this is the {\it total cut} of $\FF$. 
We shall write $\vec{v} \mmodels V(\F)$ if $\vec{v}$ is an non-total, non-empty admissible cut of $\FF$.\\

Connes and Kreimer proved in \cite{Connes} that the vector space $\h$ generated by the set of rooted forests is a Hopf algebra. Its product 
is given by the disjoint union of rooted forests, and the coproduct is defined for any rooted forest $\FF$ by:
$$\Delta(\FF)=\sum_{\vec{v} \models V(\FF)} Lea_{\vec{v}}\FF \otimes Roo_{\vec{v}}\FF
=\FF \otimes 1+1\otimes \FF+\sum_{\vec{v} \mmodels V(\FF)} Lea_{\vec{v}}\FF \otimes Roo_{\vec{v}}\FF.$$
For example:
$$\Delta\left(\tquatredeux\right)=\tquatredeux \otimes 1+1\otimes \tquatredeux+
\tun \otimes \ttroisun+\tdeux \otimes \tdeux+\tun \otimes \ttroisdeux+\tun\tun \otimes \tdeux+\tdeux \tun \otimes \tun.$$

The following coefficients will appear in corollary \ref{12}:

\begin{defi}
\cite{Brouder,Foissy3,Hoffman,Zhao}. Let $\FF$ be a rooted forest. The coefficient $\FF!$ is the integer defined by:
$$\FF!=\prod_{v \in V(\FF)} |\{w \in V(\FF)\:\mid \:w \twoheadrightarrow v\}|.$$
\end{defi}

Typical examples are given by:
$$\begin{array}{c|c|c|c|c|c|c|c|c|c|c|c|c|c|c|c}
\FF&\tun&\tun\tun&\tdeux&\tun\tun\tun&\tdeux\tun&\ttroisun&\ttroisdeux&\tun\tun\tun&
\tdeux\tun\tun&\ttroisun\tun&\ttroisdeux\tun&\tquatreun&\tquatredeux&\tquatrequatre&\tquatrecinq\\
\hline \FF!&1&1&2&1&2&3&6&1&2&3&6&4&8&12&24
\end{array}$$

\subsection{Hopf algebras of planar decorated trees}

We now recall the construction of the non-commutative generalisation of the Connes-Kreimer Hopf algebra \cite{Foissy3,Holtkamp}.\\

A {\it planar forest} is a rooted forest $\FF$ such that the set of the roots of $\FF$ is totally ordered and, for any vertex $v \in V(\FF)$, the set 
$\{w \in V(\FF)\:\mid \:w \rightarrow v\}$ is totally ordered. The set of planar forests of degree $n$ will be denoted by $\F_p(n)$ for all $n \geq 1$.\\

Planar forests are represented in such a manner that the total orders on the set of roots and the sets $\{w \in V(\FF)\:\mid \:w \rightarrow v\}$ for any $v \in V(\FF)$
are given from left to right. For example:
\begin{eqnarray*}
\F_p(0)&=&\{1\},\\
\F_p(1)&=&\{\tun\},\\
\F_p(2)&=&\{\tun\tun,\tdeux\},\\
\F_p(3)&=&\{\tun\tun\tun,\tdeux\tun,\tun\tdeux,\ttroisun,\ttroisdeux\},\\
\F_p(4)&=&\{\tun\tun\tun\tun,\tdeux\tun\tun,\tun\tdeux\tun,\tun\tun\tdeux,\tdeux\tdeux,
\ttroisun\tun,\ttroisdeux\tun,\tun\ttroisun,\tun\ttroisdeux,\tquatreun,\tquatredeux,\tquatretrois,\tquatrequatre,\tquatrecinq\}.
\end{eqnarray*}

If $\vec{v} \models V(\FF)$, then $Lea_{\vec{v}}\FF$ and $Roo_{\vec{v}}\FF$ are naturally planar forests. It is proved in \cite{Foissy3} that the space $\h_p$
generated by planar forests is a bialgebra. Its product is given by the concatenation of planar forests and its coproduct  is defined for any rooted forest $\FF$ by:
$$\Delta(\FF)=\sum_{\vec{v} \models V(\FF)} Lea_{\vec{v}}\FF \otimes Roo_{\vec{v}}\FF
=\FF \otimes 1+1\otimes \FF+\sum_{\vec{v} \mmodels V(\FF)} Lea_{\vec{v}}\FF \otimes Roo_{\vec{v}}\FF.$$
For example:
\begin{eqnarray*}
\Delta\left(\tquatredeux\right)&=&\tquatredeux \otimes 1+1\otimes \tquatredeux+
\tun \otimes \ttroisun+\tdeux \otimes \tdeux+\tun \otimes \ttroisdeux+\tun\tun \otimes \tdeux+\tdeux \tun \otimes \tun,\\
\Delta\left(\tquatretrois\right)&=&\tquatretrois \otimes 1+1\otimes \tquatretrois+
\tun \otimes \ttroisun+\tdeux \otimes \tdeux+\tun \otimes \ttroisdeux+\tun\tun \otimes \tdeux+\tun\tdeux \otimes \tun.
\end{eqnarray*}

We shall need decorated versions of this Hopf algebra. If $\D$ is any non-empty set, a {\it decorated planar forest} is a couple $(\FF,d)$,
where $\FF$ is a planar forest and $d:V(\FF) \longrightarrow \D$ is any map. The algebra of decorated planar forest $\h_p^\D$ is also a Hopf algebra.
Moreover, if $\D$ is a graded set, that is to say $\D$ is decomposed as $\displaystyle \D=\bigsqcup_{n\in \mathbb{N}} \D(n)$,
$\h_p^\D$ is naturally graded, the degree of a decorated planar forest $\D$ being the sum of the degrees of the decorations of the vertices of $\FF$.
If $\D_0=\emptyset$, the graded Hopf algebra $\h_p^\D$ is connected, and, if we define the Poincaré-Hilbert formal series:
$$f_\D(x)=\sum_{n=1}^\infty |\D(n)|x^n,\hspace{.5cm} f_{\h_p^\D}(x)=\sum_{n=0}^\infty dim\left(\h_p^\D(n)\right)x^n,$$
then:
$$f_{\h_p^\D}(x)=\frac{1-\sqrt{1-4f_\D(x)}}{2f_\D(x)},\hspace{.5cm} f_\D(x)=\frac{f_{\h_p^\D}(x)-1}{f_{\h_p^\D}(x)^2}.$$

Let us consider the graded dual of $\h_p$. The dual basis of the basis of forests is denoted by $(Z_\FF)_{\FF\in \F_p}$.
The product of two elements $Z_\FF$ and $Z_\GG$ is the sum of elements $Z_\HH$, where $\HH$ is obtained by grafting $\FF$ on $\GG$. For example:
$$Z_{\tun}Z_{\tdeux}=Z_{\tun\tdeux}+Z_{\ttroisun}+Z_{\ttroisdeux}+Z_{\ttroisun}+Z_{\tdeux\tun}=Z_{\tun\tdeux}+Z_{\tdeux\tun}+2Z_{\ttroisun}+Z_{\ttroisdeux}.$$
This product can be split into two non-associative products $\prec$ and $\succ$, such that, for all $x,y,z \in (\h_p^*)_+$:
$$\left\{\begin{array}{rcl}
(x\prec y)\prec z&=&x\prec( yz),\\
(x \succ y)\prec z&=&x \succ (y \prec z),\\
(xy)\succ z&=&x \succ (y\succ z).
\end{array}\right.$$
In other words, $(\h_p^*)_+$ is a dendriform algebra \cite{Loday,Ronco}. It is proved in \cite{Foissy2} that $(\h_p^*)_+$ is freely generated by $Z_{\tun}$, 
as a dendriform algebra. For example:
$$Z_{\tun}\prec Z_{\tdeux}=Z_{\tdeux\tun},\hspace{.5cm} Z_{\tun}\succ Z_{\tdeux}=Z_{\tun\tdeux}+2Z_{\ttroisun}+Z_{\ttroisdeux}.$$

More generally, the graded dual $(\h_p^\D)^*_+$ is the free dendriform algebra generated by the elements $Z_{\tdun{d}}$, $d\in \D$.\\

{\bf Remark.} The dual coproducts of $\prec$ and $\succ$ on $(\h_p^\D)_+$ are not the coproducts $\tdelta_\prec$ and $\tdelta_\succ$ 
introduced in section \ref{s4.1} of this text.

\subsection{Hopf algebra of ordered trees}

\begin{defi} 
An \textnormal{ordered (rooted) forest} is a rooted forest with a total order on the set of its vertices. The set of ordered forests will be denoted by $\F_o$;
for all $n \geq 0$, the set of ordered forests with $n$ vertices will be denoted by $\F_o(n)$. The $K$-vector space generated by $\F_o$ is denoted by
$\h_o$. It is a graded subspace, the homogeneous component of degree $n$ being $Vect(\F_o(n))$ for all $n \in \mathbb{N}$.
\end{defi}

{\bf Examples.}
\begin{eqnarray*}
\F_o(0)&=&\{1\},\\
\F_o(1)&=&\{\tdun{1}\},\\
\F_o(2)&=&\{\tdun{1}\tdun{2},\tddeux{1}{2},\tddeux{2}{1}\},\\
\F_o(3)&=&\left\{\tdun{1}\tdun{2}\tdun{3},
\tdun{1}\tddeux{2}{3},\tdun{1}\tddeux{3}{2},\tddeux{1}{3}\tdun{2},\tdun{2}\tddeux{3}{1},\tddeux{1}{2}\tdun{3},\tddeux{2}{1}\tdun{3},
\tdtroisun{1}{3}{2},\tdtroisun{2}{3}{1},\tdtroisun{3}{2}{1},
\tdtroisdeux{1}{2}{3},\tdtroisdeux{1}{3}{2},\tdtroisdeux{2}{1}{3},\tdtroisdeux{2}{3}{1},\tdtroisdeux{3}{1}{2},\tdtroisdeux{3}{2}{1}\right\}.
\end{eqnarray*}

{\bf Remarks.} \begin{enumerate}
\item Note that an ordered forest is also planar, by restriction of the total order to the subsets of vertices formed by the roots 
or $\{w \in V(\FF)\:\mid \:w \rightarrow v\}$.
\item We shall often identify the set $V(\FF)$ of an ordered forest $\FF$ of degree $n$ with the set $\{1,\ldots,n\}$, using the unique increasing bijection
from $V(\FF)$ to $\{1,\ldots,n\}$. 
\end{enumerate}

If $\FF$ and $\GG$ are two ordered forests, then the rooted forest $\mathbb{FG}$ is also an ordered forest with, 
for all $v \in V(\FF)$, $w \in V(\GG)$, $v<w$. This defines a non-commutative product on the set of ordered forests. 
For example, the product of $\tdun{1}$ and $\tddeux{1}{2}$ gives $\tdun{1}\tddeux{2}{3}$, whereas the product of $\tddeux{1}{2}$ and $\tdun{1}$
gives $\tddeux{1}{2}\tdun{3}=\tdun{3}\tddeux{1}{2}$. This product is linearly extended to $\h_o$, which in this way becomes a graded algebra.\\

If $\FF$ is an ordered forest, then any subforest of $\FF$ is also ordered. 
So we can define a coproduct $\Delta:\h_o \longrightarrow \h_o\otimes \h_o$ on $\h_o$ in the following way: for all $\FF \in \F_o$,
$$\Delta(\FF)=\sum_{\vec{v} \models V(\FF)} Lea_{\vec{v}}\FF\otimes Roo_{\vec{v}}\FF.$$
For example:
$$\Delta\left(\tdquatredeux{2}{3}{4}{1}\right)=\tdquatredeux{2}{3}{4}{1} \otimes 1+1\otimes \tdquatredeux{2}{3}{4}{1}
+\tdun{1} \otimes \tdtroisun{1}{3}{2}+\tddeux{2}{1} \otimes \tddeux{1}{2}+\tdun{1} \otimes \tdtroisdeux{2}{3}{1}
+\tdun{1}\tdun{2} \otimes \tddeux{1}{2}+\tddeux{3}{1}\tdun{2} \otimes \tdun{1}.$$

{\bf Remark.} This is the coopposite of the coproduct defined in \cite{FoissyUnt}. This observation will make the redaction of section 4 easier.
Note that $\h_o$ is isomorphic to $\h_o^{op}$, via the reversing of the orders on the vertices of each ordered forest; moreover, $\h_o$ is isomorphic
to $\h_o^{op,cop}$ via the antipode; so $\h_o$ is isomorphic to $\h_o^{cop}$.\\

The number of ordered forests of degree $n$ is $(n+1)^{n-1}$, see sequence A000272 of \cite{Sloane}. Hence, we have proved:

\begin{prop}
The Poincaré-Hilbert formal series of $\h_o$ is $\displaystyle f_{\h_o}(x)=\sum_{n=0}^\infty (n+1)^{n-1}x^n$.
\end{prop}

\subsection{Hopf algebra of heap-ordered trees}

\begin{defi} \cite{Grossman}
An ordered forest is \textnormal{heap-ordered} if for all $i,j \in V(\FF)$, $(i\twoheadrightarrow j)$ $\Longrightarrow$ $(i>j)$. 
The set of heap-ordered forests will be denoted by $\F_{ho}$; for all $n \geq 0$, the set of heap-ordered forests with $n$ vertices will be denoted by $\F_{ho}(n)$.
\end{defi}

For example:
\begin{eqnarray*}
\F_{ho}(0)&=&\{1\},\\
\F_{ho}(1)&=&\{\tdun{1}\},\\
\F_{ho}(2)&=&\{\tdun{1}\tdun{2},\tddeux{1}{2}\},\\
\F_{ho}(3)&=&\left\{\tdun{1}\tdun{2}\tdun{3},\tdun{1}\tddeux{2}{3},\tdun{2}\tddeux{1}{3},\tdun{3}\tddeux{1}{2},\tdtroisun{1}{2}{3},\tdtroisdeux{1}{2}{3}\right\}.
\end{eqnarray*}

If $\FF$ and $\GG$ are two heap-ordered forests, then $\mathbb{FG}$ is also heap-ordered. If $\FF$ is a heap-ordered forest, 
then any subforest of $\FF$ is heap-ordered. So the subspace $\h_{ho}$ of $\h_o$ generated by the heap-ordered forests 
is a graded Hopf subalgebra of $\h_o$.\\

Note that $\h_{ho}$ is neither commutative nor cocommutative. Indeed, $\tdun{1}.\tddeux{1}{2}=\tdun{1}\tddeux{2}{3}$ and
$\tddeux{1}{2}.\tdun{1}=\tddeux{1}{2}\tdun{3}$. Moreover:
$$\Delta(\tdtroisun{1}{3}{2})=\tdtroisun{1}{3}{2}\otimes 1+1\otimes \tdtroisun{1}{3}{2}+2\tdun{1}\otimes \tddeux{1}{2}+\tdun{1}\tdun{2}\otimes \tdun{1}.$$
So neither $\h_{ho}$ nor its graded dual $\h_{ho}^*$, are isomorphic to the Hopf algebra of heap-ordered trees of \cite{Grossman,Grossman2}, 
which is cocommutative. \\

It is well-known that the number of heap-ordered forests of degree $n$ is $n!$. Therefore, we have:

\begin{prop} 
The Poincaré-Hilbert formal series of $\h_{ho}$ is $\displaystyle f_{\h_{ho}}(x)=\sum_{n=0}^\infty n! x^n$.
\end{prop}

%
%

\section{Permutations and parking functions}

\subsection{$\FQSym$ and $\PQSym$}

We here briefly recall the construction of the Hopf algebra $\FQSym$ of free quasi-symmetric functions, also called the Malvenuto-Reutenauer Hopf algebra
\cite{Duchamp,Malvenuto}. As a vector space, a basis of $\FQSym$ is given by the disjoint union of the symmetric groups $\S_n$, for all $n \geq 0$.
We represent a permutation $\sigma \in \S_n$ by the word $(\sigma(1)\ldots\sigma(n))$. By convention, the unique element of $\S_0$ is denoted by $1$.
The product of $\FQSym$ is given, for $\sigma \in \S_k$, $\tau \in \S_l$, by:
$$\sigma.\tau=\sum_{\zeta \in Sh(k,l)}  (\sigma \otimes \tau)\circ \zeta^{-1},$$
where $Sh(k,l)$ is the set of $(k,l)$-shuffles. In other words, the product of $\sigma$ and $\tau$ is given by shifting the letters of the word
representing $\tau$ by $k$, and then summing over all the possible shufflings of this word and of the word representing $\sigma$.
For example:
\begin{eqnarray*}
(123)(21)&=&(12354)+(12534)+(15234)+(51234)+(12543)\\
&&+(15243)+(51243)+(15423)+(51423)+(54123).
\end{eqnarray*}

Let $\sigma \in \S_n$. For all $0\leq k \leq n$, there exists a unique triple 
$\left(\sigma_1^{(k)},\sigma_2^{(k)},\zeta_k\right)\in \S_k \times \S_{n-k} \times Sh(k,l)$
such that $\sigma=\zeta_k \circ \left(\sigma_1^{(k)} \otimes \sigma_2^{(k)}\right)$. The coproduct of $\FQSym$ is then defined by:
$$\Delta(\sigma)=\sum_{k=0}^n \sigma_1^{(k)} \otimes \sigma_2^{(k)}.$$
Note that $\sigma_1^{(k)}$ and $\sigma_2^{(k)}$ are obtained by cutting the word representing $\sigma$ between the $k$-th and the $(k+1)$-th letter,
and then {\it standardizing} the two obtained word, that is to say applying to their letters the unique increasing bijection to $\{1,\ldots,k\}$ or $\{1,\ldots,n-k\}$.
For example:
\begin{eqnarray*}
\Delta((41325))&=&1\otimes (41325)+Std(4)\otimes Std(1325)+Std(41)\otimes Std(325)\\
&&+Std(413)\otimes Std(25)+Std(4132)\otimes Std(5)+(41325)\otimes 1\\
&=&1\otimes (41325)+(1) \otimes (1324)+(21) \otimes (213)\\
&&+(312)\otimes (12)+(4132) \otimes (1)+(41325) \otimes (1).
\end{eqnarray*}
Then $\FQSym$ is a Hopf algebra. It is graded, with $\FQSym(n)=vect(\S_n)$ for all $n \geq 0$.
The formal series of $\FQSym$ is:
$$f_\FQSym(x)=\sum_{n=0}^\infty n!x^n=f_{\h_o}(x).$$
Moreover, $\FQSym$ has a non-degenerate Hopf pairing, homogeneous of degree $0$, defined by:
$$\langle \sigma,\tau \rangle_\FQSym=\delta_{\sigma^{-1},\tau},$$
where $\sigma$ and $\tau$ are two permutations.\\

This construction is generalized to parking functions in \cite{Novelli2,Novelli1}. A {\it parking function} of degree $n$ is a word $(a_1,\ldots,a_n)$,
of $n$ letters in $\mathbb{N}^*$, such that in the reordered word $(a'_1,\ldots,a'_n)$ satisfies $a_i'\leq i$ for all $i$. 
For example, permutations are parking functions. Here are the parking functions of degree $\leq 3$:
\begin{eqnarray*}
&&1,\\
&&(1),\\
&&(12),\:(21),\:(11),\\
&&(123),\:(132),\:(213),\:(231),\:(312),\:(321),\\
&&(112),\:(121),\:(211),\:(113),\:(131),\:(311),\:(122),\:(212),\:(221),\:(111).
\end{eqnarray*}

If $\sigma=(a_1,\ldots,a_k)$ and $\tau=(b_1,\ldots,b_l)$, we define the parking function $\sigma \otimes \tau$ by:
$$\sigma \otimes \tau=(a_1,\ldots,a_k,b_1+k,\ldots,b_l+k).$$
Considering parking functions as maps from $\{1,\ldots,n\}$ to $\mathbb{N}^*$, we define a product on the space $\PQSym$ generated by parking functions by:
$$\sigma.\tau=\sum_{\zeta \in Sh(k,l)} (\sigma \otimes \tau) \circ \zeta^{-1},$$
where $\sigma$ and $\tau$ are parking functions of respective degrees $k$ and $l$.
For example:
\begin{eqnarray*}
(121)(11)&=&(12144)+(12414)+(14214)+(41214)+(12441)\\
&&+(14241)+(41241)+(14421)+(41421)+(44121).
\end{eqnarray*}
The standardization is extended to parking functions and $\PQSym$ inherits a coproduct similar to the coproduct of $\FQSym$.
Moreover, $\FQSym$ is a Hopf subalgebra of $\PQSym$.

\subsection{From ordered forests to permutations}

We recall the following result of \cite{FoissyUnt}:

\begin{prop}
Let $n \geq 0$. For all $\FF \in \F_o(n)$, let $S_\FF$ be the set of permutations $\sigma \in \S_n$ such that for all $1\leq i,j \leq n$,
($i \twoheadrightarrow j$) $\Longrightarrow$ ($\sigma^{-1}(i)\geq \sigma^{-1}(j)$). Let us define:
$$\Theta:\left\{\begin{array}{rcl}
\h_o&\longrightarrow&\FQSym\\
\FF \in \F_o&\longrightarrow&\displaystyle \sum_{\sigma \in S_\FF} \sigma.
\end{array}\right.$$
Then $\Theta:\h_o^{cop}\longrightarrow \FQSym$ is a Hopf algebra morphism, homogeneous of degree $0$.
\end{prop}

For example, if $\{a,b,c\}=\{1,2,3\}$:
\begin{eqnarray*}
\Theta(\tdun{1})&=&(1),\\
\Theta(\tdun{1}\tdun{2})&=&(12)+(21),\\
\Theta(\tddeux{1}{2})&=&(12),\\
\Theta(\tddeux{2}{1})&=&(21),\\
\Theta(\tdun{$a$}\tdun{$b$}\tdun{$c$})&=&(abc)+(acb)+(bac)+(bca)+(cab)+(cba),\\
\Theta(\tdun{$a$}\tddeux{$b$}{$c$})&=&(abc)+(bac)+(bca),\\
\Theta(\tdtroisun{$a$}{$b$}{$c$})&=&(abc)+(acb),\\
\Theta(\tdtroisdeux{$a$}{$b$}{$c$})&=&(abc).
\end{eqnarray*}

{\bf Remark.} Note that $\Theta$ can also be seen as a Hopf algebra morphism from $\h_o$ to $\FQSym^{cop}$.\\

The morphism $\Theta$ is not injective, for example $\Theta(\tddeux{1}{2}+\tddeux{2}{1}-\tdun{1}\tdun{2})=0$. From \cite{FoissyUnt}, we recall:

\begin{prop} \label{7}
The restriction of $\Theta$ to $\h_{ho}^{cop}$ is an isomorphism of graded Hopf algebras.
\end{prop}

Can we extend this construction from ordered forests to parking functions, in order to obtain a Hopf algebra isomorphism from $\h_o^{cop}$ to $\PQSym$?
The answer is given by the following result:

\begin{prop}
Let $\Theta':\h_o^{cop} \longrightarrow \PQSym$ be a Hopf algebra morphism, homogeneous of degree $0$. We assume that for all $\FF \in \F_o$, 
there exists a set $S'_{\FF}$ of parking functions such that $\Theta'(\FF)=\displaystyle \sum_{\sigma \in S'_{\FF}}\sigma$. 
Then $\Theta'$ is not an isomorphism.
\end{prop}

\begin{proof} As $\Theta'$ is homogeneous of degree $0$, if $\FF \in \F_o(n)$, then $S'_{\FF}\subseteq \PQSym(n)$, so is a set of parking functions
of size $n$. So $S'_{\tdun{1}}=\{(1)\}$. In particular:
$$\Theta'(\tdun{1}\tdun{2})=\Theta'(\tdun{1})\Theta'(\tdun{1})=(1)(1)=(12)+(21),$$
so $S'_{\tdun{1}\tdun{2}}=\{(12),(21)\}$.
Moreover:
$$\Delta^{op}(\tddeux{1}{2})=\tddeux{1}{2}\otimes 1+1\otimes \tddeux{1}{2}+\tdun{1}\otimes \tdun{1},\hspace{.5cm}
\Delta^{op}(\tddeux{2}{1})=\tddeux{2}{1}\otimes 1+1\otimes \tddeux{2}{1}+\tdun{1}\otimes \tdun{1}.$$
So $S'_{\tddeux{1}{2}}$ and $S'_{\tddeux{2}{1}}$ are equal to $\{(12)\}$, $\{(21)\}$ or $\{(11)\}$.
If they are equal, then $\Theta'$ is not injective. Let us assume that they are different. If both are equal to $\{(12)\}$ or $\{(21)\}$, then 
$\Theta'(\tddeux{1}{2}+\tddeux{2}{1})=(12)+(21)=\Theta'(\tdun{1}\tdun{2})$, so $\Theta'$ is not injective. It remains four cases:
\begin{itemize}
\item $S'_{\tddeux{1}{2}}=\{(12)\}$ and $S'_{\tddeux{2}{1}}=\{(11)\}$. Then:
$$\tdelta\circ \Theta'(\tdtroisdeux{2}{3}{1})=(\Theta' \otimes \Theta')\circ \tdelta^{op}(\tdtroisdeux{2}{3}{1})=(12)\otimes (1)+(1)\otimes (11),$$
so the only possibility is $S'_{\tdtroisdeux{2}{3}{1}}=\{(122)\}$. Similarly:
$$\tdelta\circ \Theta'(\tdtroisdeux{1}{3}{2})=(\Theta' \otimes \Theta')\circ \tdelta^{op}(\tdtroisdeux{1}{3}{2})=(12)\otimes (1)+(1)\otimes (11),$$
so the only possibility is $S'_{\tdtroisdeux{1}{3}{2}}=\{(122)\}$. Hence, $\Theta'$ is not injective.

\item  $S'_{\tddeux{1}{2}}=\{(11)\}$ and $S'_{\tddeux{2}{1}}=\{(12)\}$. Similarly, considering $\tdtroisdeux{3}{1}{2}$ and $\tdtroisdeux{2}{1}{3}$, we conclude
that $\Theta'$ is not injective.

\item $S'_{\tddeux{1}{2}}=\{(21)\}$ and $S'_{\tddeux{2}{1}}=\{(11)\}$. Then:
$$\tdelta\circ \Theta' (\tdtroisdeux{2}{1}{3})=\tdelta \circ \Theta'(\tdtroisdeux{3}{1}{2})=(11)\otimes (1)+(1)\otimes (21).$$
So $S'_{\tdtroisdeux{2}{1}{3}}=S'_{\tdtroisdeux{3}{1}{2}}=\{(221)\}$: $\Theta'$ is not injective.

\item  $S'_{\tddeux{1}{2}}=\{(11)\}$ and $S'_{\tddeux{2}{1}}=\{(21)\}$. Similarly, considering $\tdtroisdeux{1}{3}{2}$ and $\tdtroisdeux{2}{3}{1}$, we conclude
that $\Theta'$ is not injective.
\end{itemize}
Therefore, $\Theta'$ is never injective. \end{proof} \\

{\bf Remark}. It is possible to prove a similar result for Hopf algebra morphisms from $\h_o$ to $\PQSym$.

\section{Pairing on $\h_o$}

\subsection{Definition}

\begin{theo}
For all $\FF,\GG \in \F_o$, we put:
$$S(\FF,\GG)=\left\{f:V(\FF)\longrightarrow V(\GG),\: bijective\:\mid \: 
\substack{\displaystyle \forall x,y \in V(\FF),\: (x \twoheadrightarrow y) \Longrightarrow (f(x)\geq f(y))\\[2mm]
\displaystyle \forall x,y \in V(\FF), \:(f(x) \twoheadrightarrow f(y))\Longrightarrow (x \geq y)}\right\}.$$
We define a pairing on $\h_o^{cop}$ by $\langle \FF,\GG\rangle=|S(\FF,\GG)|$. This pairing is Hopf, symmetric, and homogeneous of degree $0$.
\end{theo}

\begin{proof}
If $\FF$ and $\GG$ do not have the same number of vertices, then $S(\FF,\GG)=\emptyset$, so $\langle \FF,\GG\rangle=0$: the pairing is homogeneous.
Moreover, the map $f\longrightarrow f^{-1}$ is a bijection from $S(\FF,\GG)$ to $S(\GG,\FF)$ for any $\FF,\GG \in \F_o$, so the pairing is symmetric.\\

We now prove that $\langle \FF_1\FF_2,\GG \rangle=\langle \FF_1 \otimes \FF_2,\Delta^{op}(\GG)\rangle$ for any forests $\FF_1,\FF_2,\GG \in \F_o$. 
We consider $f\in S(\FF_1\FF_2,\GG)$. Let $x'\in f(V(\FF_2))$ and $y' \in V(\GG)$, such that $y' \twoheadrightarrow x'$. 
As $f$ is bijective, there exists $x \in V(\FF_2)$, $y\in V(\FF_1\FF_2)$, such that $f(x)=x'$ and $f(y)=y'$. As $f \in S(\FF,\GG)$, $y \geq x$ in $\FF$. 
Since $x\in V(\FF_2)$, $y \in V(\FF_2)$, it follows that $y' \in f(V(\FF_2))$. Hence, there exists a unique admissible cut $\vec{v}_f \models V(\GG)$,
such that $Lea_{\vec{v}_f}(\GG)=f(V(\FF_2))$ and $Roo_{\vec{v}_f}(\GG)=f(V(\FF_1))$. 
Moreover, $f_{\mid V(\FF_1)}\in S(\FF_1,Roo_{\vec{v}_f}(\GG))$ and $f_{\mid V(\FF_2)}\in S(\FF_2,Lea_{\vec{v}_f}(\GG))$.
Hence, this defines a map:
$$\upsilon:\left\{\begin{array}{rcl}
S(\FF_1\FF_2,\GG)&\longrightarrow&\displaystyle \bigsqcup_{\vec{v} \models V(\GG)} S(\FF_1,Roo_{\vec{v}}(\GG))\times S(\FF_2,Lea_{\vec{v}}(\GG))\\
f&\longrightarrow&(f_{\mid V(\FF_1)},f_{\mid V(\FF_2)}).
\end{array}\right.$$
It is clearly injective. Let us show that it is surjective. Let $\vec{v}\models V(\GG)$, $(f_1,f_2) \in S(\FF_1,Roo_{\vec{v}}(\GG))\times S(\FF_2,Lea_{\vec{v}}(\GG))$.
Let $f:V(\FF_1\FF_2)\longrightarrow V(\GG)$ be the unique bijection such that $f_{\mid V(\FF_i)}=f_i$ for $i=1,2$. 
Let us show that $f \in S(\FF_1\FF_2,\GG)$. 

If $x\twoheadrightarrow y$ in $\FF_1\FF_2$, then $x \twoheadrightarrow y$ in $\FF_i$, for $i=1$ or $2$. So $f_i(x) \geq f_i(y)$ and $f(x) \geq f(y)$.

Let us assume that $f(x) \twoheadrightarrow f(y)$ in $\GG$. Three cases are then possible:
\begin{itemize}
\item $f(x) \twoheadrightarrow f(y)$ in $Roo_{\vec{v}}(\GG)$: then $f(x)=f_1(x)$, $f(y)=f_1(y)$. So $x \geq y$ in $\FF_1$, so $x \geq y$ in $\FF_1\FF_2$.
\item $f(x) \twoheadrightarrow f(y)$ in $Lea_{\vec{v}}(\GG)$: similar proof.
\item $f(x) \in Lea_{\vec{v}}(\GG)$ and $f(y) \in Roo_{\vec{v}}(\GG)$: so $y\in V(\FF_1)$ and $x \in V(\FF_2)$, so $x \geq y$ in $\FF_1\FF_2$.
\end{itemize}
We conclude that $\upsilon$ is a bijection. So the following equation holds:
$$\langle \FF_1\FF_2,\GG \rangle=|S(\FF_1\FF_2,\GG)|=\sum_{\vec{v}\models V(\GG)} |S(\FF_1,Roo_{\vec{v}}(\GG))||S(\FF_2,Lea_{\vec{v}}(\GG))|
=\langle \FF_1 \otimes \FF_2,\Delta^{op}(\GG)\rangle.$$
As it is symmetric, the pairing is a Hopf pairing. \end{proof}\\

We list  the matrices of the pairing in degree $1$ and $2$:
$$\begin{array}{c|c}
&\tdun{1}\\
\hline \tdun{1}&1
\end{array} \hspace{1cm}
\begin{array}{c|c|c|c}
&\tdun{1}\tdun{2}&\tddeux{1}{2}&\tddeux{2}{1}\\
\hline\tdun{1}\tdun{2}&2&1&1\\
\hline\tddeux{1}{2}&1&1&0\\
\hline\tddeux{2}{1}&1&0&1
\end{array}$$

This pairing is degenerate. For example, $\tddeux{1}{2}+\tddeux{2}{1}-\tdun{1}\tdun{2}$ is in the kernel of the pairing.

\subsection{Kernel of the pairing}

\begin{lemma}
Let $\FF,\GG \in \F_o(n)$. The elements of $S_\FF$, $S_\GG$ and $S(\FF,\GG)$ can all be seen as elements of $\S_n$,
identifying $V(\FF)$ and $V(\GG)$ with $\{1,\ldots,n\}$, using the unique increasing bijections from $V(\FF)$ or $V(\GG)$ to $\{1,\ldots,n\}$. 
Then $S(\FF,\GG)=S_\FF^{-1} \cap S_\GG$.
\end{lemma}

\begin{proof} Let $f \in S(\FF,\GG)$. If $x\twoheadrightarrow y$ in $\FF$, then $f(x) \geq f(y)$ in $\GG$, so $f(x) \geq f(y)$ in $\{1,\ldots,n\}$,
so $f(x) \geq f(y)$ in $\FF$. So $f^{-1} \in S_\FF$ and $f \in S_\FF^{-1}$.
If $x' \twoheadrightarrow y'$ in $\GG$, then $f^{-1}(x) \geq f^{-1}(y)$ in $\FF$, so in $\{1,\ldots,n\}$, so in $\GG$. Hence, $f \in S_\GG$.

Let $f \in S_\FF^{-1} \cap S_\GG$. If $x \twoheadrightarrow y$ in $\FF$, as $f^{-1} \in S_\FF$, $f(x) \geq f(y)$ in $\FF$, so in $\{1,\ldots,n\}$,
so in $\GG$. If $f(x) \twoheadrightarrow f(y)$ in $\GG$, then as $f \in S_\GG$, $x \geq y$ in $\GG$, so in $\{1,\ldots,n\}$, 
so in $\FF$. Hence, $f \in S(\FF,\GG)$. \end{proof}

\begin{prop}
For any $x,y \in \h_o$, $\langle x,y \rangle=\langle \Theta(x),\Theta(y)\rangle_\FQSym$.
\end{prop}

\begin{proof} It is enough to take $x=\FF$, $y=\GG$ in $\F_o$. Then:
$$\langle \Theta(\FF),\Theta(\GG)\rangle_\FQSym=\sum_{\sigma \in S_\FF,\tau \in S_\GG} \langle \sigma,\tau\rangle
=\sum_{\sigma \in S_\FF,\tau \in S_\GG} \delta_{\sigma^{-1},\tau}=|S_\FF^{-1} \cap S_\GG|=|S(\FF,\GG)|=\langle \FF,\GG\rangle.$$
So $\Theta$ respects the pairings. \end{proof}

\begin{cor} \label{12}
The kernel of the pairing on $\h_o$ is $Ker(\Theta)$. Moreover, the restriction of the pairing to $\h_{ho}$ is non-degenerate.
\end{cor}

\begin{proof} $\Theta_{\mid \h_{ho}}$ is an isometry from $\h_{ho}$ to $\FQSym$. As the pairing of $\FQSym$ is non-degenerate,
the same holds for the pairing of $\h_{ho}$. Moreover, for any $x \in \h_o$, as $\Theta$ is surjective:
\begin{eqnarray*}
x \in \h_o^\perp&\Longleftrightarrow&\forall y\in \h_o, \: \langle x,y\rangle=0\\
&\Longleftrightarrow&\forall y\in \h_o, \: \langle \Theta(x),\Theta(y)\rangle_\FQSym=0\\
&\Longleftrightarrow&\forall y'\in \FQSym, \: \langle \Theta(x),y'\rangle=0\\
&\Longleftrightarrow&\Theta(x) \in \FQSym^\perp\\
&\Longleftrightarrow& \Theta(x)=0.
\end{eqnarray*}
So $\h_o^\perp=Ker(\Theta)$. \end{proof}\\

The next proposition gives application of the pairing:

\begin{prop}
For all $n\geq 0$, for all $\FF \in \F_o(n)$, $\displaystyle |S_\FF|=\frac{|\FF|!}{\FF!}=\langle \tdun{$1$}\ldots \tdun{$n$},\FF \rangle$.
\end{prop}

\begin{proof} Let us fix $n \geq 0$. The symmetric group $\S_n$ naturally acts on $\F_o(n)$ by permutation of the orders of the vertices of the ordered forests.
For example, if $\sigma \in \S_3$:
\begin{equation*}\sigma.\tdtroisun{1}{3}{2}=\hspace{4mm}\tdtroisun{$\sigma(1)$}{$\sigma(3)$}{\hspace{-4mm}$\sigma(2)$}\hspace{1.4mm},\hspace{.5cm}
\sigma.\tdtroisdeux{1}{2}{3}=\tdtroisdeux{$\sigma(1)$}{$\sigma(2)$}{$\sigma(3)$}\hspace{4mm}.
\end{equation*}
Let $\FF \in \F_o(n)$, $\sigma \in \S_n$. For any bijection $f:V(\FF)\longrightarrow \{1,\ldots,n\}$:
\begin{eqnarray*}
f \in S_{\sigma.\FF}&\Longleftrightarrow&\forall i,j \in V(\FF),\: (i\twoheadrightarrow j \mbox{ in }\sigma.\FF) \Longleftrightarrow (f^{-1}(i) \geq f^{-1}(j))\\
&\Longleftrightarrow&\forall i,j \in V(\FF),\: (\sigma(i) \twoheadrightarrow \sigma(j) \mbox{ in }\sigma.\FF) \Longleftrightarrow 
(f^{-1}\circ \sigma(i) \geq f^{-1}\circ \sigma(j))\\
&\Longleftrightarrow&\forall i,j \in V(\FF),\: (i \twoheadrightarrow j \mbox{ in }\FF) \Longleftrightarrow 
(f^{-1}\circ \sigma(i) \geq f^{-1}\circ \sigma(j))\\
&\Longleftrightarrow& \sigma^{-1}\circ f \in S_\FF.
\end{eqnarray*}
So $S_{\sigma.\FF}=\sigma\circ S_\FF$. As a consequence, $|S_\FF|$ does not depend of the order of the vertices of $\FF$, 
but only of the subjacent rooted forest.\\

It is clear that $S_{\tdun{$1$}\ldots \tdun{$n$}}=\S_n$, so $S(\tdun{$1$}\ldots \tdun{$n$},\FF)=\S_n^{-1}\cap S_\FF=S_\FF$. Hence,
$\langle \tdun{$1$}\ldots \tdun{$n$},\FF \rangle=|S_\FF|$.
Let us now prove that $\langle \tdun{$1$}\ldots \tdun{$n$},\FF \rangle=\frac{|\FF|!}{\FF!}$ by induction on the degree $n$ of $\FF$.
If $n=0$, this is obvious. Let us assume the result for any forest of degree $<n$. Two cases can occur.
\begin{itemize}
\item $\FF$ is not connected. As $\langle \tdun{$1$}\ldots \tdun{$n$},\FF \rangle=|S_\FF|$ does not depend of the order of the vertices of $\FF$, we can assume
that $\FF=\FF_1\FF_2$, with $deg(\FF_1)=n_1,deg(\FF_2)=n_2$, $n_1,n_2<n$. Then:
\begin{eqnarray*}
\langle \tdun{$1$}\ldots \tdun{$n$},\FF\rangle&=&\langle \Delta^{op}(\tdun{$1$}\ldots \tdun{$n$}),\FF_1\otimes \FF_2\rangle\\
&=&\sum_{i+j=n}\frac{n!}{i!j!} \langle \tdun{$1$}\ldots \tdun{$i$}\otimes \tdun{$1$}\ldots \tdun{$j$},\FF_1\otimes \FF_2\rangle\\
&=&0+\frac{n!}{n_1!n_2!}\langle \tdun{$1$}\ldots \tdun{$n_1$}\:\otimes \tdun{$1$}\ldots \tdun{$n_2$}\:,\FF_1\otimes \FF_2\rangle\\
&=&\frac{n!}{n_1!n_2!}\frac{n_1!}{\FF_1!}\frac{n_2!}{\FF_2!}\\
&=&\frac{n!}{\FF_1!\FF_2!}\\
&=&\frac{n!}{\FF!}.
\end{eqnarray*}

\item $\FF$ is connected. There is only one admissible cut $\vec{v} \models V(\FF)$, such that $Roo_{\vec{v}}\FF$ is of degree $1$:
$\vec{v}=\{w\in V(\FF)\:\mid\: w \rightarrow r\}$, where $r$ is the root of $\FF$. Then $\FF!=nLea_{\vec{v}}\FF!$. So:
$$\langle \tdun{$1$}\ldots \tdun{$n$},\FF \rangle=\langle\tdun{$1$} \otimes \tdun{$1$}\ldots \tdun{$n-1$}\hspace{.4cm}, \Delta^{op}(\FF)\rangle
=\langle \tdun{$1$},\tdun{$1$}\rangle \langle \tdun{$1$}\ldots \tdun{$n-1$}\hspace{.4cm}, Lea_{\vec{v}}\FF\rangle+0=1\frac{(n-1)!}{Lea_{\vec{v}}\FF!}
=\frac{n!}{\FF!}.$$
\end{itemize}\end{proof}

\section{An isomorphism from ordered forests to parking functions}

\subsection{Other structures on $\h_p^\D$}

\label{s4.1} If $F,G$ are two non-empty planar forests, eventually decorated, we denote by $F \nwarrow G$ the planar forest, eventually decorated,
obtained by grafting $G$ on the leaf of $F$ that is at most on the right. This defines a product $\nwarrow$ on $(\h_p^\D)_+$, the augmentation ideal of $\h_p^\D$.  \\

{\bf Examples.}  In the non-decorated case:
$$\begin{array}{|rclcl|rclcl|rclcl|rclcl|}
\hline \tun\tun\tun &\nwarrow& \tdeux&=&\tun\tun\ttroisdeux &\tdeux &\nwarrow& \tun \tun \tun&=&\tcinqdix&
\tun\tun&\nwarrow&\tun\tun\tun&=&\tun \tquatreun& \tun\tun\tun&\nwarrow&\tun\tun&=&\tun\tun\ttroisun\\
\tun\tdeux &\nwarrow& \tdeux&=&\tun\tquatrecinq& \tdeux &\nwarrow& \tun \tdeux&=&\tcinqdouze&
\tun\tun&\nwarrow&\tun\tdeux&=&\tun\tquatretrois& \tun\tdeux&\nwarrow&\tun\tun&=&\tun\tquatrequatre\\
\tdeux\tun &\nwarrow& \tdeux&=&\tdeux\ttroisdeux& \tdeux &\nwarrow& \tdeux\tun &=&\tcinqonze&
\tun\tun&\nwarrow&\tdeux\tun&=&\tun\tquatredeux& \tdeux\tun&\nwarrow&\tun\tun&=&\tdeux\ttroisun\\
\ttroisun &\nwarrow& \tdeux&=&\tcinqneuf& \tdeux &\nwarrow& \ttroisun&=&\tcinqtreize&
\tun\tun&\nwarrow&\ttroisun&=&\tun\tquatrequatre& \ttroisun&\nwarrow&\tun\tun&=&\tcinqsept\\
\ttroisdeux &\nwarrow& \tdeux&=&\tcinqquatorze& \tdeux &\nwarrow& \ttroisdeux&=&\tcinqquatorze&
\tun\tun&\nwarrow& \ttroisdeux&=&\tun\tquatrecinq& \ttroisdeux&\nwarrow&\tun\tun&=&\tcinqtreize\\
\hline \end{array}$$

The following properties are easily verified for $x,y,z$ non-empty forests:

\begin{lemma}
For all $x,y,z \in (\h_p^\D)_+$:
$$\left\{\begin{array}{rcl}
(x \nwarrow y)\nwarrow z&=&x \nwarrow (y \nwarrow z),\\
(xy)\nwarrow z&=&x (y \nwarrow z).
\end{array}\right.$$
\end{lemma}

We recover the definition of \cite{Loday2}:

\begin{defi}
A  {\it duplicial algebra} is a triple $(A,.,\nwarrow)$, where $A$ is a vector space and $.,\nwarrow:A\otimes A \longrightarrow A$, with the following axioms:
for all $x,y,z \in A$,
\begin{equation}\label{E1}\left\{\begin{array}{rcl}
(xy)z&=&x(yz),\\
(x \nwarrow y)\nwarrow z&=&x \nwarrow (y \nwarrow z),\\
(xy)\nwarrow z&=&x (y \nwarrow z).
\end{array}\right. \end{equation}
\end{defi}

{\bf Remark.} If $(A,m,\nwarrow)$ is a duplicial algebra, then $A^{op}=(A,m^{op},\nwarrow^{op})$ is a $\P_\nearrow$-algebra, as defined in \cite{Foissy}.
In particular, for $A=\h_p$, the $\P_\nearrow$-algebra $\h_p^{op}$ is isomorphic to $(\h_p,m,\nearrow)$, where $\FF \nearrow \GG$ is defined in \cite{Foissy}
by grafting $\FF$ on the leaf at most on the left of $\GG$. An explicit isomorphism is given by sending a planar forets $\FF$ to its image by a vertical symmetry. \\

Here is an alternative description of the free duplicial algebras:

\begin{prop}
For all set $\D$, $(\h_p^\D)_+$ is the free duplicial algebra generated by the elements $\tdun{$d$}$'s, $d\in \D$.
\end{prop}

\begin{proof} In order to simplify the proof, we only treat here the case where $\D$ is reduced to a single element, 
that is to say we work with non-decorated planar forests. The general proof is very similar. Let $A$ be a duplicial algebra and let $a \in A$. 
Let us prove there exists a unique morphism of duplicial algebras $\phi:(\h_p)_+\longrightarrow A$, such that $\phi(\tun)=a$. 
We define $\phi(\FF)$ for any non-empty planar forest $\FF$ inductively on the degree of $\FF$ by:
$$\left\{\begin{array}{rcl}
\phi(\tun)&=&a,\\
\phi(t_1\ldots t_k)&=&\phi(t_1)\ldots \phi(t_k) \mbox{ if }k\geq 2,\\
\phi(B^+(F))&=&a\nwarrow \phi(F). 
\end{array}\right.$$
As the product of $A$ is associative, this is perfectly defined. This map is linearly extended into a map $\phi:(\h_p)_+\longrightarrow A$.
Let us show it is a morphism of duplicial algebras. By the second point, $\phi(xy)=\phi(x)\phi(y)$
for any forests $x,y \in (\h_p)_+$. Let $x,y$ be two non-empty forests. Let us prove that $\phi(x\nwarrow y)=\phi(x) \nwarrow \phi(y)$ by induction on 
$n=deg(x)$. If $n=1$, then $x=\tun$, so:
$$\phi(x \nwarrow y)=\phi(B^+(y))=a \nwarrow \phi(y)=\phi(x)\nwarrow \phi(y).$$
Let us assume the result for any forest of weight $<n$. We put $x=t_1\ldots t_k$, $t_k=B^+(F)$. Then, using the induction hypothesis on $F$:
\begin{eqnarray*}
\phi(x \nwarrow y)&=&\phi(t_1\ldots t_{k-1} B^+(F \nwarrow y))\\
&=&\phi(t_1)\ldots \phi(t_{k-1}) (a \nwarrow \phi((F \nwarrow y)))\\
&=&\phi(t_1)\ldots \phi(t_{k-1}) (a \nwarrow (\phi(F) \nwarrow \phi(y)))\\
&=&\phi(t_1)\ldots \phi(t_{k-1}) ((a \nwarrow \phi(F)) \nwarrow \phi(y))\\
&=&\phi(t_1)\ldots \phi(t_{k-1}) (\phi(t_k) \nwarrow \phi(y))\\
&=&(\phi(t_1)\ldots \phi(t_{k-1})\phi(t_k)) \nwarrow \phi(y))\\
&=&\phi(x) \nwarrow \phi(y).
\end{eqnarray*}
We use the convention $1 \nwarrow y=y$, if $t_k=\tun$. So $\phi$ is a morphism of duplicial algebras. \\

Let $\phi':(\h_p)_+\longrightarrow A$ be another morphism of duplicial algebras such that $\phi'(\tun)=a$. 
Then for any planar trees $t_1,\ldots,t_k$, $\phi'(t_1\ldots t_k)=\phi'(t_1)\ldots \phi'(t_k)$.
For any planar forest $F$, $\phi'(B^+(F))=\phi'(\tun \nwarrow F)=a \nwarrow \phi'(F)$. So $\phi=\phi'$.  \end{proof}\\

\begin{defi}
For any non-empty planar forest $\FF$, let $r_\FF$ be the leaf of $\FF$ that is at most on the right. We put:
$$\tdelta_\prec(\FF)=\sum_{\substack{\vec{v}\mmodels V(\FF) \\ r_\FF \in Lea_{\vec{v}}\FF}} Lea_{\vec{v}} \FF \otimes Roo_{\vec{v}} \FF,\hspace{.5cm}
\tdelta_\succ(\FF)=\sum_{\substack{\vec{v}\mmodels V(\FF) \\ r_\FF \in Roo_{\vec{v}}\FF}} Lea_{\vec{v}} \FF \otimes Roo_{\vec{v}} \FF.$$
\end{defi}

Note that $\tdelta_\prec+\tdelta_\succ=\tdelta$. 

\begin{lemma}
For any $x \in (\h_p^\D)_+$:
\begin{equation} \label{E2} \left\{\begin{array}{rcl}
(\tdelta_\prec\otimes Id)\circ \tdelta_\prec(x)&=&(Id \otimes \tdelta)\circ \tdelta_\prec(x),\\
(\tdelta_\succ\otimes Id)\circ \tdelta_\prec(x)&=&(Id \otimes \tdelta_\prec)\circ \tdelta_\succ(x),\\
(\tdelta\otimes Id)\circ \tdelta_\succ(x)&=&(Id \otimes \tdelta_\succ)\circ \tdelta_\succ(x).
\end{array}\right. \end{equation}
In other words, $(\h_p^\D)_+$ is a dendriform coalgebra.
\end{lemma}

\begin{proof} It is enough to prove this statement if $x$ is a non-empty forest. We put, as $\tdelta$ is coassociative:
$$(\tdelta \otimes Id) \circ \tdelta(x)=(Id \otimes \tdelta)\circ \tdelta(x)=\sum x^{(1)} \otimes x^{(2)} \otimes x^{(3)},$$
where the $x^{(1)}$, $x^{(2)}$, $x^{(3)}$ are subforests of $x$. Then:
$$\left\{\begin{array}{rcccl}
(\tdelta_\prec\otimes Id)\circ \tdelta_\prec(x)&=&(Id \otimes \tdelta)\circ \tdelta_\prec(x)
&=&\displaystyle \sum_{r_x \in x^{(1)}} x^{(1)} \otimes x^{(2)}\otimes x^{(3)},\\
(\tdelta_\succ\otimes Id)\circ \tdelta_\prec(x)&=&(Id \otimes \tdelta_\prec)\circ \tdelta_\succ(x)
&=&\displaystyle \sum_{r_x \in x^{(2)}} x^{(1)} \otimes x^{(2)}\otimes x^{(3)},\\
(\tdelta\otimes Id)\circ \tdelta_\succ(x)&=&(Id \otimes \tdelta_\succ)\circ \tdelta_\succ(x)
&=&\displaystyle \sum_{r_x \in x^{(3)}} x^{(1)} \otimes x^{(2)}\otimes x^{(3)}.
\end{array}\right.$$
So $(\h_p^\D)_+$ is a dendriform coalgebra. \end{proof}\\

{\bf Notations}. \begin{enumerate}
\item If $(A,\tdelta_\prec,\tdelta_\succ)$ is a dendriform coalgebra, we denote $Prim_{tot}(A)=Ker(\tdelta_\prec)\cap Ker(\tdelta_\succ)$.
\item Let $(A,\tdelta_\prec,\tdelta_\succ)$ be a dendriform coalgebra. We shall use the following sweedler notations: for any $a\in A$,
$\tdelta(a)=a'\otimes a''$, $\tdelta_\prec(a)=a'_\prec \otimes a''_\prec$ and $\tdelta_\succ(a)=a'_\succ \otimes a''_\succ$.
\end{enumerate}

\begin{prop} \label{19}
The dendriform coalgebra $(\h_p^\D)_+$ is freely cogenerated by the elements $\tdun{$d$}$'s, $d\in \D$.
As a consequence, $Prim_{tot}\left((\h_p^\D)_+\right)=Vect(\tdun{$d$},\:d\in \D)$.
\end{prop}

\begin{proof} It is equivalent to prove that the graded dual $(\h_p^\D)_+^*$ is the free dendriform algebra generated by the elements $Z_{\tdun{$d$}}$'s,
$d\in \D$.  In order to simplify the proof, we only treat here the case where $\D$ is reduced to a single element, 
that is to say we work with non-decorated planar forests. Comparing the Poincaré-Hilbert formal series, it is enough to prove that $(\h_p)_+^*$ 
is generated by $Z_{\tun}$. For any forests $F,G$, we have:
$$Z_F \prec Z_G=\sum_{\substack{\mbox{\scriptsize $H$ grafting of $F$ on $G$}\\ r_H=r_F}}Z_H,\hspace{.5cm}
Z_F \succ Z_G=\sum_{\substack{\mbox{\scriptsize $H$ grafting of $F$ on $G$}\\ r_H=r_G}}Z_H.$$
In particular, $Z_F \succ Z_{\tun}=Z_{F\tun}$.
 
Let $A$ be the (associative) subalgebra of $(\h_p)^*_+$ generated by the elements $Z_{F\tun}$, $F$ planar forest. Let us prove that $Z_G \in A$
for any non-empty planar forest $G$ by induction on $n=deg(G)$. If $n=1$, $Z_G=Z_{\tun} \in A$. Let us assume that any $Z_H \in A$, if $deg(H)<n$.
If $deg(G)=n$, we put $G=t_1\ldots t_{k-1}B^+(H)$. We proceed by induction on $deg(H)=l$.
If $l=0$, then $G=t_1\ldots t_{k-1}\tun$, so $Z_G \in A$. If $l\geq 1$, then:
$$Z_H Z_{t_1\ldots t_{k-1}\tun}=Z_G+R,$$
where $R$ is a sum of $Z_{F'}$, with $F'$ of weight $n$, of the form $F'=t'_1\ldots t'_r B^+(H')$, $deg(H')<l$. By the induction hypothesis on $l$,
$R \in A$. By the induction hypothesis on $n$, $Z_H \in A$; moreover, $Z_{t_1\ldots t_{k-1}\tun}\in A$. So $Z_G \in A$.\\

Let $B$ the dendriform subalgebra of $(\h_p)^*_+$ generated by $Z_{\tun}$. Let us prove that $B=(\h_p)^*_+$. 
By the first point, it is enough to prove that for any planar forests $F_1,\ldots,F_k$, $x=Z_{F_1\tun}\ldots Z_{F_k \tun} \in B$. 
We proceed by induction on $deg(x)=deg(F_1)+\ldots+deg(F_k)+k$. If $n=1$, then $x=Z_{\tun} \in B$.
If $n \geq 2$, then the induction hypothesis gives $Z_{F_1},\ldots,Z_{F_k} \in B$. Then:
$$x=(Z_{F_1} \succ Z_{\tun})\ldots (Z_{F_k} \succ Z_{\tun}) \in B.$$
So $(\h_p)^*_+$ is generated by $Z_{\tun}$. \end{proof}

\begin{prop}\label{20}
\begin{enumerate}
\item Let $x,y\in (\h_p^\D)_+$. Then:
\begin{equation}\label{E3}\left\{\begin{array}{rcl}
\tdelta_\prec(xy)&=&y \otimes x+x'y\otimes x''+xy'_\prec \otimes y''_\prec+y'_\prec \otimes x y''_\prec+x'y'_\prec \otimes x''y''_\prec,\\[2mm]
\tdelta_\succ(xy)&=&x \otimes y+x'\otimes x''y+x y'_\succ\otimes y''_\succ+y'_\succ \otimes x y''_\succ+x'y'_\succ \otimes x''y''_\succ.
\end{array}\right. \end{equation}
In other words, $(\h_p^\D)_+$ is a codendriform bialgebra.

\item Let $x,y \in (\h_p^\D)_+$. Then:
\begin{equation}\label{E4}\left\{\begin{array}{rcl}
\tdelta_\prec(x \nwarrow y)&=&y \otimes x+y'_\prec \otimes x \nwarrow y''_\prec+x'_\prec \nwarrow y \otimes x''_\prec\\[1mm]
&&+x'_\succ y \otimes x''_\succ+x'_\succ y'_\prec \otimes x''_\succ\nwarrow x''_\prec,\\[2mm]
\tdelta_\succ(x \nwarrow y)&=&y'_\succ \otimes x \nwarrow y''_\succ+x'_\succ \otimes x''_\succ \nwarrow y
+x'_\succ y'_\succ \otimes x''_\succ \nwarrow y''_\succ.
\end{array}\right. \end{equation} \end{enumerate} \end{prop}

\begin{proof} It is enough to prove these formulas if $x=\FF$, $y=\GG$ are non-empty planar forests. Let us first compute $\tdelta_\prec(\FF\GG)$.
For any admissible cut $\vec{v}\mmodels V(\FF\GG)$, let $\vec{v}'$ be the restriction of $\vec{v}$ to $\FF$ and $\vec{v}''$ the restriction of $\vec{v}$ to $\GG$. 
Then $\vec{v}'\models V(\FF)$ and $\vec{v}'' \models V(\GG)$. Moreover, $\vec{v}'$ and $\vec{v}''$ are not simultaneously total, and not simultaneously empty. \\

Let us first compute $\tdelta_\prec(\FF\GG)$. Let $\vec{v} \mmodels V(\FF\GG)$, such that $r_{\FF\GG}=r_\GG$ belongs to $Lea_{\vec{v}}\FF\GG$.
So $r_\GG \in Lea_{\vec{v}''}\GG$, so $\vec{v}''$ is not empty. There are five possibilities for $\vec{v}$:
\begin{itemize}
\item $\vec{v}'$ is empty and $\vec{v}''$ is total: this gives the term $\GG\otimes \FF$. 
\item $\vec{v}'$ is not empty and $\vec{v}''$ is total: then $\vec{v}' \mmodels V(\FF)$, and this gives the term $\FF'\GG\otimes \FF''$.
\item $\vec{v}'$ is empty and $\vec{v}''$ is not total: as $r_\GG \in Lea_{\vec{v}''}\GG$, this gives the term $\GG'_\prec \otimes \FF \GG''_\prec$.
\item $\vec{v}'$ is total and $\vec{v}''$ is not total: as $r_\GG \in Lea_{\vec{v}''}\GG$, this gives the term $\FF\GG'_\prec \otimes \GG''_\prec$.
\item $\vec{v}' \mmodels V(\FF)$ and $\vec{v}''$ is not total: as $r_\GG \in Lea_{\vec{v}''}\GG$, this gives the term $\FF'\GG'_\prec \otimes \FF''\GG''_\prec$.
\end{itemize}

We now compute $\tdelta_\succ(\FF\GG)$.  Let $\vec{v} \mmodels V(\FF\GG)$, such that $r_{\FF\GG}=r_\GG$ belongs to $Roo_{\vec{v}}\FF\GG$.
So $r_\GG \in Roo_{\vec{v}''}\GG$, so $\vec{v}''$ is not total. There are five possibilities for $\vec{v}$:
 \begin{itemize}
\item $\vec{v}'$ is total and $\vec{v}''$ is empty: this gives the term $\FF\otimes \GG$. 
\item $\vec{v}'$ is not total and $\vec{v}''$ is empty: then $\vec{v}' \mmodels V(\FF)$, and this gives the term $\FF'\otimes \FF''\GG$.
\item $\vec{v}'$ is total and $\vec{v}''$ is not empty: as $r_\GG \in Roo_{\vec{v}''}\GG$, this gives the term $\FF\GG'_\succ\otimes \GG''_\succ$.
\item $\vec{v}'$ is empty and $\vec{v}''$ is not total: as $r_\GG \in Roo_{\vec{v}''}\GG$, this gives the term $\GG'_\succ \otimes \FF\GG''_\succ$.
\item $\vec{v}' \mmodels V(\FF)$ and $\vec{v}''$ is not total: as $r_\GG \in Roo_{\vec{v}''}\GG$, this gives the term $\FF'\GG'_\succ \otimes \FF''\GG''_\succ$.
\end{itemize}

For any admissible cut $\vec{v}\mmodels V(\FF\nwarrow \GG)$, let $\vec{v}'$ be the restriction of $\vec{v}$ to $\FF$ and let $\vec{v}''$ be the unique admissible cut
of $\GG$ such that $Lea_{\vec{v}''}\GG$ is the subforest of $Lea_{\vec{v}} \FF\nwarrow \GG$ formed by the vertices that belong to $V(\GG)$. 
Moreover, if $\vec{v}'$ is not empty, as $\vec{v}$ is admissible, $r_\FF \in Lea_{\vec{v}'}\FF$, if, and only if $\vec{v}''$ is total. 
Consequently, as $\vec{v}$ is not total, $\vec{v}'$ is not total.\\

We compute $\tdelta_\prec(\FF\nwarrow \GG)$.  Let $\vec{v} \mmodels V(\FF\nwarrow \GG)$, such that $r_{\FF\nwarrow \GG}=r_\GG$ belongs to 
$Lea_{\vec{v}}\FF\nwarrow \GG$. As $r_\GG \in Lea_{\vec{v}''}\GG$, $\vec{v}''$ is not empty. There are four possibilities for $\vec{v}$:
\begin{itemize}
\item $\vec{v}'$ is empty and $\vec{v}''$ is total: this gives the term $\GG \otimes \FF$.
\item $\vec{v}'$ is empty and $\vec{v}''$ is not total: then $\vec{v}'' \mmodels V(\GG)$ and $r_\GG \in Lea_{\vec{v}''}\GG$, 
so this gives the term $\GG'_\prec \otimes \FF \nwarrow \GG''_\prec$.
\item $\vec{v}'$ is not empty and $\vec{v}''$ is total: we obtain two subcases:
\begin{itemize}
\item $Lea_{\vec{v}'}\FF$ contains $r_\FF$: this gives the term $\FF'_\prec \nwarrow \GG \otimes \FF''_\prec$.
\item $Roo_{\vec{v}'}\FF$ contains $r_\FF$: this gives the term $\FF'_\succ \GG \otimes \FF''_\succ$.
\end{itemize}
\item $\vec{v}'$ is not empty and $\vec{v}''$ is not total: then $r_\FF$ does not belong to $Lea_{\vec{v}'}\FF$, $r_\GG$ belongs to $Lea_{\vec{v}''}\GG$,
and this gives the term $\FF'_\succ \GG'_\prec \otimes \FF''_\succ \nwarrow \GG''_\prec$.
\end{itemize}

Finally, we compute $\tdelta_\succ(\FF\nwarrow \GG)$.  Let $\vec{v} \mmodels V(\FF\nwarrow \GG)$, such that $r_{\FF\nwarrow \GG}=r_\GG$ belongs to
$Roo_{\vec{v}}\FF\nwarrow \GG$. As $r_\GG \in Roo_{\vec{v}''}\GG$, $\vec{v}''$ is not total. So $\vec{v}'$ does not contain $r_\FF$. 
There are three possibilities for $\vec{v}$:
\begin{itemize}
\item $\vec{v}'$ is empty: then $\vec{v}'' \mmodels V(\GG)$ and $Roo_{\vec{v}''} \GG$ contains $r_\GG$, and this gives the term
$\GG'_\succ \otimes \FF\nwarrow \GG''_\succ$.
\item $\vec{v}'$ is not empty and $\vec{v}''$ is empty: then $r_\FF \in Roo_{\vec{v}'}\FF$ and we obtain the term $\FF'_\succ \otimes \FF''_\succ \nwarrow \GG$.
\item $\vec{v}'$ is not empty and $\vec{v}''$ is not empty:  then $r_\FF \in Roo_{\vec{v}'}\FF$, $r_\GG \in Roo_{\vec{v}''}\GG$ and we obtain the term
 $\FF'_\succ \GG'_\succ \otimes \FF''_\succ \nwarrow \GG'_\succ$.
\end{itemize} \end{proof}

This suggests:

\begin{defi} A $Dup$-$Dend$ bialgebra is a family $(A,.,\nwarrow,\tdelta_\prec,\tdelta_\succ)$, where $A$ is a vector space,
$.,\nwarrow:A\otimes A\longrightarrow A$ and $\tdelta_\prec,\tdelta_\succ:A \longrightarrow A \otimes A$, with the following properties:
\begin{enumerate}
\item $(A,.,\nwarrow)$ is a duplicial algebra (axioms \ref{E1}).
\item $(A,\tdelta_\prec,\tdelta_\succ)$ is a dendriform coalgebra (axioms \ref{E2}).
\item The compatibilities of proposition \ref{20} are satisfied (axioms \ref{E3} and \ref{E4}).
\end{enumerate} \end{defi}

\subsection{A rigidity theorem}

\begin{theo}
Let $A$ be a $Dup$-$Dend$ bialgebra. We assume that $A$ is graded and connected, that is to say $A_0=(0)$.
Let $(p_\D)_{d\in \D}$ be a basis of $Prim_{tot}(A)$ formed by homogeneous elements, indexed by a graded set $\D$. There exists a unique isomorphism
of graded $Dup$-$Dend$ bialgebras:
$$\phi:\left\{\begin{array}{rcl}
(\h_p^\D)_+&\longrightarrow&A\\
\tdun{$d$},\:d\in \D&\longrightarrow&p_d.
\end{array}\right.$$
\end{theo}

\begin{proof} The graded dual $A^*$ of $A$ is a graded dendriform algebra, so is the quotient of a free dendriform algebra. By proposition \ref{19}, 
we can find a graded set $\D'$, such that there exists a epimorphism of dendriform algebras $\pi:(\h_p^{\D'})_+^*\longrightarrow A^*$. 
Dually, we obtain a monomorphism of dendriform coalgebras $\iota':A\longrightarrow (\h_p^{\D'})_+$. 
For all $d\in \D$, $\iota'(p_d) \in Prim_{tot}((\h_p^{\D'})_+)$, so is a linear span of $\tdun{$d'$}$, $d'\in \D'$. Up to an automorphism of $(\h_p^{\D'})_+$, 
we can assume that $\D \subseteq \D'$ and that $\iota'(p_d)=\tdun{$d$}$ for all $d\in \D$. Composing $\iota'$ with the canonical epimorphism 
from $(\h_p^{\D'})_+$ to $(\h_p^\D)_+$ obtained by deleting the forests with vertices with decorations which are not in $\D$, 
we obtain a morphism of dendriform coalgebras $\iota:A\longrightarrow (\h_p^\D)_+$, sendind $\iota(p_d)$ to $\tdun{$d$}$ for all $d\in \D$.
If $\iota$ is not injective, let us consider $x \in Ker(\iota)$, non-zero, of minimal degree. As $\iota$ is a morphism of dendriform coalgebras, $x\in Prim_{tot}(A)$:
absurd, $\iota$ is clearly injective on $Prim_{tot}(A)$. So $\iota$ is injective. \\

Let $x \in A$. We denote by $IC_n$ the set of $n$-iterated coproducts of $A$:
\begin{eqnarray*}
IC_0&=&\{Id\},\\
IC_n&=&\left\{\begin{array}{c}
Id^{\otimes i} \otimes \tdelta_\prec \otimes Id^{\otimes (n-i-1)}\circ \Theta,\:
Id^{\otimes i-1} \otimes \tdelta_\succ \otimes Id^{\otimes (n-i)}\circ \Theta \:\mid\\
 \: \Theta \in IC_{n-1},\:1\leq i \leq n\end{array}\right\}.
\end{eqnarray*}
The elements of $IC_n$ are homogeneous maps from $A$ to $A^{\otimes (n+1)}$. As a consequence, if $x \in A_n$, $k\geq n$ and $\Omega \in IC_k$, 
then $\Theta(x) \in \left(A^{\otimes (k+1)}\right)_n=(0)$. So for any $x \in A$, there exists a greatest integer $k$ such that there exists $\Omega \in IC_k$,
$\Omega(x) \neq 0$. This $k$ will be denoted by $deg_p(x)$.\\

Let us prove that $A$ is generated by $Prim_{tot}(A)$. Let $B$ be the $\P_\nwarrow$-subalgebra of $A$ generated by $Prim_{tot}(A)$.
Let $x \in A$, we prove that $x \in B$ by induction on $deg_p(x)$. If $deg_p(x) \leq 1$, then $x \in prim_{tot}(A)$ and the result is obvious.
Let us assume that any $y\in A$, such that $deg_p(y)<k$, is in $B$, and let us take $x \in A$, such that $deg_p(x)=k$.
Then $\iota(x)$ is an element of $(\h_p^\D)_+$, so can be written as an expression in $\tdun{$d$}$, 
using the products $.$ and $\nwarrow$ of $(\h_p^\D)_+$: we put $\iota(x)=P(\tdun{$d$},\:d\in \D)$. Let us consider $y=P(p_d,\:d\in \D)$.
As $deg_p(x)=k$, for any $\Omega \in IC_{k-1}$, $\Omega(x) \in Prim_{tot}(A)^{\otimes k}$, as it is cancelled by
$Id^{\otimes i} \otimes \tdelta_\prec \otimes Id^{\otimes (k-i-1)}$ and $Id^{\otimes i} \otimes \tdelta_\succ \otimes Id^{\otimes (k-i-1)}$ for any $i$.
We put:
$$\Omega(x)=\sum_{d_1,\ldots,d_k \in \D} a_{d_1,\ldots,d_k} p_{d_1}\otimes \ldots \otimes p_{d_k}.$$
As $\iota$ is a morphism of coalgebras:
$$\Omega(\iota(x))=(\iota^{\otimes k})\circ \Omega(x)=\sum_{d_1,\ldots,d_k \in \D} a_{d_1,\ldots,d_k} \tdun{$d_1$}\otimes \ldots \otimes \tdun{$d_k$}\:.$$
As the compatibilities between the products and the coproducts are the same in $A$ and $(\h_p^\D)_+$, we deduce that:
$$(\iota^{\otimes k})\circ \Omega(y)=\Omega(\iota(y))=\sum_{d_1,\ldots,d_k \in \D} a_{d_1,\ldots,d_k} \tdun{$d_1$}\otimes \ldots \otimes 
\tdun{$d_k$}=\Omega(\iota(x)).$$
The injectivity of $\iota$ implies that $\Omega(x-y)=0$. So the induction hypothesis can be applied on $x-y$, which belongs to $B$. 
By definition, $y \in B$, so $x \in B$.\\

As a consequence, the unique  morphism of duplicial algebras $\phi:(\h_p^\D)_+\longrightarrow A$,
sending $\tdun{$d$}$ to $p_d$ for all $d\in \D$, is surjective.  Moreover, it is homogeneous of degree $0$, as $\tdun{$d$}$ and $p_d$
are homogeneous of the same degree for all $d \in \D$. As it sends $\tdun{$d$}$ to an element of $Prim_{tot}(A)$ for all $d$, 
$\phi$ is a morphism of $Dup$-$Dend$ bialgebras. If $\phi$ is not injective, let us take $x\in Ker(\phi)$, non-zero, of minimal degree. 
As $\phi$ is a morphism of dendriform coalgebras, $x\in Prim_{tot}((\h_p^\D)_+)=Vect(\tdun{$d$},\:d\in\D)$: absurd, by definition the restriction of $\phi$ 
on $Vect(\tdun{$d$},\:d\in\D)$ is injective. So $\phi$ is an isomorphism.  \end{proof}

\subsection{Application to ordered forests}

We define the following product $\nwarrow$ on $(\h_o)_+$ for two non-empty ordered forests $\FF$ and $\GG$ in the following way:
$\FF \nwarrow \GG$ is the ordered forests obtained by grafting $\GG$ on the greatest vertex of $\FF$, the vertices of $\FF$ being smaller than the vertices of $\GG$
in the grafting. For example:
$$\tddeux{$1$}{$2$}\nwarrow \tdun{$1$}\tdun{$2$}=\tdquatrequatre{$1$}{$2$}{$4$}{$3$},\hspace{.5cm}
\tddeux{$2$}{$1$}\nwarrow \tdun{$1$}\tdun{$2$}=\tdquatreun{$2$}{$4$}{$3$}{$1$}.$$

Let $\FF,\GG,\HH$ be three non-empty ordered forests. As the greatest vertex of $\FF\GG$ is also the greatest vertex of $\GG$, 
$(\FF\GG)\nwarrow \HH=\FF(\GG \nwarrow \HH)$.
As the greatest vertex of $\FF\nwarrow \GG$ is the greatest vertex of $\GG$, $(\FF\nwarrow \GG)\nwarrow \HH=\FF\nwarrow (\GG \nwarrow \HH)$.
So $((\h_o)_+,.,\nwarrow)$ is a duplicial algebra. \\

For any ordered forest $\FF$, we denote by $g_\FF$ the greatest vertex of $\FF$. We then put:
$$\tdelta_\prec(\FF)=\sum_{\substack{\vec{v} \mmodels V(\FF)\\ g_\FF \in V(Lea_{\vec{v}}\FF)}} Lea_{\vec{v}}\FF \otimes Roo_{\vec{v}} \FF,\hspace{.5cm}
\tdelta_\succ(\FF)=\sum_{\substack{\vec{v} \mmodels V(\FF)\\ g_\FF \in V(Roo_{\vec{v}}\FF)}} Lea_{\vec{v}}\FF \otimes Roo_{\vec{v}} \FF.$$
For example:
$$\Delta_\prec\left(\tdquatredeux{2}{3}{4}{1}\right)=\tddeux{2}{1} \otimes \tddeux{1}{2}+\tddeux{3}{1}\tdun{2} \otimes \tdun{1},\hspace{.5cm}
\Delta_\succ\left(\tdquatredeux{2}{3}{4}{1}\right)=\tdun{1} \otimes \tdtroisdeux{2}{3}{1}+\tdun{1} \otimes \tdtroisun{1}{3}{2}+\tdun{1}\tdun{2} \otimes \tddeux{1}{2}.$$

Then $((\h_o)_+,\tdelta_\prec,\tdelta_\succ)$ is a dendriform coalgebra: indeed, if $\FF$ is a non-empty ordered forest, we put:
$$(\tdelta \otimes Id) \circ \tdelta(\FF)=(Id \otimes \tdelta)\circ \tdelta(\FF)=\sum \FF^{(1)} \otimes \FF^{(2)} \otimes \FF^{(3)},$$
where the $\FF^{(1)}$, $\FF^{(2)}$ and $\FF^{(3)}$ are subforests of $\FF$. Then:
$$\left\{\begin{array}{rcccl}
(\tdelta_\prec\otimes Id)\circ \tdelta_\prec(\FF)&=&(Id \otimes \tdelta)\circ \tdelta_\prec(\FF)
&=&\displaystyle \sum_{g_\FF \in \FF^{(1)}} \FF^{(1)} \otimes \FF^{(2)}\otimes \FF^{(3)},\\
(\tdelta_\succ\otimes Id)\circ \tdelta_\prec(\FF)&=&(Id \otimes \tdelta_\prec)\circ \tdelta_\succ(\FF)
&=&\displaystyle \sum_{g_\FF \in \FF^{(2)}} \FF^{(1)} \otimes \FF^{(2)}\otimes \FF^{(3)},\\
(\tdelta\otimes Id)\circ \tdelta_\succ(\FF)&=&(Id \otimes \tdelta_\succ)\circ \tdelta_\succ(\FF)
&=&\displaystyle \sum_{g_\FF \in \FF^{(3)}} \FF^{(1)} \otimes \FF^{(2)}\otimes \FF^{(3)}.
\end{array}\right.$$

Moreover, one can prove similarly with proposition \ref{20} that $(\h_o)_+$ is a $Dup$-$Dend$ bialgebra. 
Moreover, $(\h_{ho})_+$ is clearly a sub-$Dup$-$Dend$ bialgebra of $(\h_o)_+$. Hence:

\begin{theo}
\begin{enumerate}
\item There exists a graded set $\D_o$, such that $(\h_o)_+$ is isomorphic to $(\h_p^{\D_o})_+$ as graded $Dup$-$Dend$ bialgebras.
\item There exists a graded set $\D_{ho}$, such that $(\h_{ho})_+$ is isomorphic to $(\h_p^{\D_{ho}})_+$ as graded $Dup$-$Dend$ bialgebras.
\end{enumerate}
\end{theo}

The formal series of $\D_o$ and $\D_{ho}$ are given by:
$$f_{\D_o}(x)=\frac{f_{\h_o}(x)-1}{f_{\h_o}(x)^2},\hspace{.5cm} f_{\D_{ho}}(x)=\frac{f_{\h_{ho}}(x)-1}{f_{\h_{ho}}(x)^2}.$$
This gives the following examples:
$$\begin{array}{c|c|c|c|c|c|c|c|c}
k&1&2&3&4&5&6&7&8\\
\hline |\D_o(k)|&1&1&7&66&786&11\:278&189\:391&3\:648\:711\\
\hline|\D_{ho}(k)|&1&0&1&6&39&284&2\:305&20\:682
\end{array}$$
These are sequences A122705 and A122827 of \cite{Sloane}.\\

{\bf Remark.} As a consequence, $\h_o$ and $\h_{ho}$ are free and cofree. It is not difficult to show that $\h_o$ is freely generated by
indecomposable ordered forests, that is to say ordered forests $\FF$ that cannot be written as $\FF=\GG \HH$, with $\GG,\HH \neq 1$.
Similarly, $\h_{ho}$ is freely generated by indecomposable heap-ordered forests.

\subsection{Application to parking functions}

Let $\sigma$ be a parking function. We denote by $m_\sigma$ the maximal index $i$ such that $\sigma(i)$ is maximal.
In particular, if $\sigma \in \S_n$, $m_\sigma=\sigma^{-1}(n)$. For any parking function $\sigma$ of degree $n\geq 1$, we put:
$$\tdelta_\prec(\sigma)=\sum_{k=1}^{m_\sigma-1} \sigma^{(k)}_2 \otimes \sigma^{(k)}_1, \hspace{.5cm}
\tdelta_\succ(\sigma)=\sum_{k=m_\sigma}^{n-1} \sigma^{(k)}_2 \otimes \sigma^{(k)}_1.$$
For example:
$$\Delta_\prec((21332))=(1332) \otimes (1)+(221)\otimes (21)+(21) \otimes (213),\hspace{.5cm}\Delta_\succ((21332))=(1) \otimes (2133).$$
Note that $\tdelta_\prec+\tdelta_\succ=\tdelta^{op}$. Moreover, for any $\sigma$ of degree $n$, we denote:
$$(\tdelta \otimes Id) \circ \tdelta(\sigma)=(Id \otimes \tdelta) \circ \tdelta(x)
=\sum_{1\leq i<j \leq n-1} \sigma^{(i,j)}_1 \otimes \sigma^{(i,j)}_2 \otimes \sigma^{(i,j)}_3,$$
where $\sigma^{(i,j)}_1 \otimes \sigma^{(i,j)}_2 \otimes \sigma^{(i,j)}_3$ is the unique term of $(\tdelta \otimes Id) \circ \tdelta(\sigma)$
with $\sigma^{(i,j)}_1$ a parking function of degree $i$,  $\sigma^{(i,j)}_2$ a parking function of degree $j-i$, and
 $\sigma^{(i,j)}_3$ a parking function of degree $n-j$. Then, we have the equalities:
$$\left\{\begin{array}{rcccl}
(\tdelta_\prec\otimes Id)\circ \tdelta_\prec(\sigma)&=&(Id \otimes \tdelta)\circ \tdelta_\prec(\sigma)
&=&\displaystyle \sum_{j< m_\sigma} \sigma^{(i,j)}_3 \otimes \sigma^{(i,j)}_2 \otimes \sigma^{(i,j)}_1,\\
(\tdelta_\succ\otimes Id)\circ \tdelta_\prec(\sigma)&=&(Id \otimes \tdelta_\prec)\circ \tdelta_\succ(\sigma)
&=&\displaystyle \sum_{i<m_\sigma \leq j} \sigma^{(i,j)}_3 \otimes \sigma^{(i,j)}_2 \otimes \sigma^{(i,j)}_1,\\
(\tdelta\otimes Id)\circ \tdelta_\succ(\sigma)&=&(Id \otimes \tdelta_\succ)\circ \tdelta_\succ(\sigma)
&=&\displaystyle \sum_{i \geq m_\sigma} \sigma^{(i,j)}_3 \otimes \sigma^{(i,j)}_2 \otimes \sigma^{(i,j)}_1.
\end{array}\right.$$
So $\PQSym^{cop}_+$ is a dendriform coalgebra. Thinking of $\sigma.\tau$ as a sum of shufflings of two words, it is easy to see that 
$\PQSym^{cop}_+$ is a codendriform bialgebra. \\

Let $\sigma,\tau$ be two parking functions of respective degrees $k$ and $l$. We put:
$$\sigma \nwarrow \tau=\sum_{\substack{\zeta \in Sh(k,l)\\ \zeta(k+1) \geq \zeta(m_\sigma)}} (\sigma \otimes \tau) \circ \zeta^{-1}.$$
In other words, $\sigma \nwarrow \tau$ is the sum of the shufflings of the word $\sigma$ and the shifted word $\tau$ shifted by $k$,
such that the letters of $\tau$ are all after the greatest letter that is more on the right, of $\sigma$. In particular, if $m_\sigma=k$, 
then $\sigma \nwarrow \tau=\sigma \otimes \tau$. For example:
$$(21331)\nwarrow(12)=(2133167)+(2133617)+(2133671).$$
So, if $\sigma,\tau,\upsilon$ are three parking functions of respective degrees $k,l,m$, as $m_{\sigma.\tau}=deg(\sigma)+m_\tau$:
$$\left\{\begin{array}{rcccl}
(\sigma.\tau) \nwarrow \upsilon&=&\sigma.(\tau \nwarrow \upsilon)
&=&\displaystyle \sum_{\substack{\zeta \in Sh(k,l,m)\\ \zeta(k+l+1) \geq \zeta(m_\tau+k)}} (\sigma \otimes \tau \otimes \upsilon) \circ \zeta^{-1},\\
(\sigma \nwarrow \tau)\nwarrow \upsilon&=&\sigma \nwarrow (\tau \nwarrow \upsilon) &=&\displaystyle 
\sum_{\substack{\zeta \in Sh(k,l,m)\\ \zeta(k+l+1) \geq \zeta(m_\tau+k)\\ \zeta(k+1) \geq \zeta(m_\sigma)}} (\sigma \otimes \tau \otimes \upsilon) \circ \zeta^{-1}.
\end{array}\right.$$
So $\PQSym^{cop}_+$ is a duplicial algebra. It is not difficult to show that $\PQSym^{cop}_+$ is a $Dup$-$Dend$ bialgebra.
Moreover, $\FQSym^{cop}_+$ is a sub-$Dup$-$Dend$ bialgebra of $\PQSym^{cop}_+$. So:

\begin{theo}
\begin{enumerate}
\item There exists a graded set $\D_P$, such that $\PQSym_+$ is isomorphic to $(\h_p^{\D_P})_+$ as graded $Dup$-$Dend$ bialgebras.
\item There exists a graded set $\D_F$, such that $\FQSym_+$ is isomorphic to $(\h_p^{\D_F})_+$ as graded $Dup$-$Dend$ bialgebras.
\end{enumerate}
\end{theo}

As $\h_o$ and $\PQSym$ have the same formal series, $\D_o$ and $\D_P$ have the same formal series, so 
$(\h_p^{\D_o})_+$ and $(\h_p^{\D_P})_+$ are isomorphic. The same result holds for $(\h_p^{\D_{ho}})_+$ and $(\h_p^{\D_F})_+$.
As a conclusion:

\begin{cor}
The graded Hopf algebras $\h_o$, $\h_p^{\D_o}$ and $\PQSym$ are isomorphic.
The graded Hopf algebras $\h_{ho}$, $\h_p^{\D_{ho}}$ and $\FQSym$ are isomorphic.
\end{cor}

\subsection{Compatibilities with $\theta$}

\begin{prop}
The morphism $\Theta:\h_o\longrightarrow \PQSym^{cop}$ is a morphism of $Dup$-$Dend$ bialgebras.
\end{prop}

\begin{proof} Let $\FF$ and $\GG$ be two non-empty ordered forest, of respective degrees $k$ and $l$. We first show:
$$S_{\FF \nwarrow \GG}=\bigsqcup_{\sigma \in S_\FF,\:\tau \in S_\GG} \bigsqcup_{\substack{\zeta \in Sh(k,l)\\ \zeta(k+1)\geq
\zeta(\sigma^{-1}(k))}}\{(\sigma \otimes \tau) \circ \zeta^{-1}\}.$$
We denote by $S'$ the set in the second member of this equation. We first prove that $S_{\FF \nwarrow \GG}\subseteq S'$.
Let $\chi  \in S_{\FF \nwarrow \GG}$. There exists a unique $(\sigma,\tau,\zeta)\in \S_k \times \S_l\times Sh(k,l)$, such that 
$\chi=(\sigma \otimes \tau) \circ \zeta^{-1}$. Let us prove that $\sigma \in S_\FF$. If $i \twoheadrightarrow j$ in $\FF$,
then $i\twoheadrightarrow j$ in $\FF \nwarrow \GG$, so:
\begin{eqnarray*}
\chi^{-1}(i)&\geq& \chi^{-1}(j),\\
\zeta \circ (\sigma^{-1} \otimes \tau^{-1})(i) &\geq& \zeta \circ (\sigma^{-1} \otimes \tau^{-1})(j),\\
\zeta \circ \sigma^{-1}(i) &\geq&\zeta\circ \sigma^{-1}(j),\\
\sigma^{-1}(i)&\geq&\sigma^{-1}(j),
\end{eqnarray*}
as $\zeta$ is increasing on $\{1,\ldots,k\}$. So $\sigma \in S_\FF$. Similarly, $\tau \in S_\GG$. Moreover, 
the vertex $\tau(1)+k$ belongs to $\GG$ in $\FF \nwarrow \GG$, so $\tau(1)+k \twoheadrightarrow k$ in $\FF\nwarrow \GG$. As a consequence:
\begin{eqnarray*}
\chi^{-1}(\tau(1)+k)&\geq& \chi^{-1}(k),\\
\zeta \circ (\sigma^{-1} \otimes \tau^{-1})(\tau(1)+k) &\geq& \zeta \circ (\sigma^{-1} \otimes \tau^{-1})(k),\\
\zeta(k+1)&\geq&\zeta\circ \sigma^{-1}(k).
\end{eqnarray*}

Let us now prove that $S'\subseteq S_{\FF \nwarrow \GG}$.  Let $\sigma \in S_\FF$, $\tau \in S_\GG$ and $\zeta \in Sh(k,l)$, 
such that $\zeta(k+1)\geq \zeta(\sigma^{-1}(k))$. We put $\chi=(\sigma \otimes \tau) \circ \zeta^{-1}$. 
Let $i,j$ be two vertices of $\FF \nwarrow \GG$, such that $i\twoheadrightarrow j$. Three cases can occur:
\begin{itemize}
\item $i,j$ are vertices of $\FF$. Then $\sigma^{-1}(i)\geq \sigma^{-1}(j)$, so $(\sigma^{-1}\otimes \tau^{-1})(i) 
\geq (\sigma^{-1}\otimes \tau^{-1})(j) $, 
and finally $\sigma^{-1}(i)=\zeta \circ(\sigma^{-1}\otimes \tau^{-1})(i) \geq \zeta \circ (\sigma^{-1}\otimes \tau^{-1})(j)=\sigma^{-1}(j)$.
\item $i,j$ are vertices of $\GG$. The same proof holds.
\item $i$ is a vertex of $\GG$ and $j$ is a vertex of $\FF$. Then $i \twoheadrightarrow k$ in $\FF \nwarrow \GG$.
As $\FF \nwarrow \GG$ is a forest, necessarily $k \twoheadrightarrow j$ in $\FF$, so by the first point $\sigma^{-1}(k)\geq \sigma^{-1}(j)$.

Moreover,$i+1 \geq k+1$, so $\sigma^{-1}(i) \geq \zeta(k+1)$ as $\zeta$ is increasing on $\{k+1,\ldots,k+l\}$. Then:
$$\sigma^{-1}(i) \geq \zeta(k+1)\geq \zeta(\sigma^{-1}(k))=\sigma^{-1}(k) \geq \sigma^{-1}(j).$$
\end{itemize}

Finally, for any non-empty ordered forests $\FF$ and $\GG$ of respective degrees $k$ and $l$:
$$\Theta(\FF \nwarrow \GG)=\sum_{\sigma \in S_\FF,\:\tau \in S_\GG} \sum_{\substack{\zeta \in Sh(k,l)\\ \zeta(k+1)\geq \zeta(\sigma^{-1}(k))}}
(\sigma \otimes \tau) \circ \zeta^{-1}=\sum_{\sigma \in S_\FF,\:\tau \in S_\GG} \sigma \nwarrow \tau=\Theta(\FF) \nwarrow \Theta(\GG).$$

Let $\FF\in \F_o(n)$. We proved in \cite{FoissyUnt} that there exists a bijection:
$$\Upsilon: \left\{\begin{array}{rcl}
S_\FF\times \{1,\ldots,n-1\}&\longmapsto&\displaystyle \bigsqcup_{\vec{v} \mmodels V(\FF)} S_{Roo_{\vec{v}}\FF}
\times S_{Lea_{\vec{v}}\FF}\\
(\sigma,k)&\longmapsto&\left(\sigma_1^{(k)},\sigma_2^{(k)}\right),
\end{array}\right.$$
where the pair $\left(\sigma_1^{(k)},\sigma_2^{(k)}\right)$ belongs to the term of the union indexed by the unique admissible cut $\vec{v}$ 
such that the vertices of $Roo_{\vec{v}}\FF$ are $\{\sigma(1),\ldots,\sigma(k)\}$, and the vertices of $Lea_{\vec{v}}\FF$ are
$\{\sigma(k+1),\ldots,\sigma(n)\}$. So, if $(\sigma,k) \in S_\FF\times \{1,\ldots,n-1\}$, 
$k<\sigma^{-1}(n)$ if, and only if, $n$ is a vertex of $Lea_{\vec{v}}\FF$. So:
$$(\Theta \otimes \Theta) \circ \tdelta_\prec(\FF)=\sum_{\substack{\vec{v} \mmodels V(\FF)\\ n \in Lea_{\vec{v}}\FF}}
\sum_{\substack{\sigma \in S_{Lea_{\vec{v}}\FF}\\\tau \in S_{Roo_{\vec{v}}\FF}}} \sigma \otimes \tau
=\sum_{\sigma \in S_\FF} \sum_{k=1}^{\sigma^{-1}(n)-1} \sigma_2^{(k)} \otimes \sigma_1^{(k)}
=\sum_{\sigma \in S_\FF} \tdelta_\prec(\sigma)=\tdelta \circ \Theta(\FF).$$
The proof is similar to show that $(\Theta \otimes \Theta) \circ \tdelta_\succ=\tdelta_\succ \circ \Theta$. \end{proof}\\

{\bf Remark.} We also obtain that the $Dup$-$Dend$ bialgebras $\h_{ho}$ and $\FQSym$ are isomorphic.

\bibliographystyle{amsplain}
\bibliography{biblio}

\end{document}